\documentclass[a4paper,10pt]{article}

\usepackage{amsmath}
\usepackage{amssymb}
\usepackage{graphicx}
\usepackage{tikz}
\usepackage{enumerate}
\usepackage{hyperref}
\usepackage{cleveref}
\usepackage{soul}
\usepackage{enumitem}

\newtheorem{theorem}{Theorem}[section]
\newtheorem{lemma}{Lemma}[section]
\newtheorem{corollary}{Corollary}[section]
\newtheorem{prop}{Proposition}[section]
\newtheorem{definition}{Definition}[section]
\newtheorem{remark}{Remark}

\newenvironment{proof}{\noindent{\textsc{Proof.}}}
{$\hfill\Box$\vspace{0.1 cm}\\}

\allowdisplaybreaks[4]

\newcommand{\R}{\mathbb R}

\newcommand{\N}{{\mathbb N}}

\newcommand{\norm}[1]{\left\| #1 \right\|}
\newcommand{\abs}[1]{\left\vert #1 \right\vert}

\newcommand{\wc}{\rightharpoonup}

\newcommand{\pt}{\partial_t\,}

\newcommand{\LL}[1]{\mathbf{L^#1}}

\newcommand{\HH}[1]{\mathbf{H^#1}}
\newcommand{\dd}[1]{\mathinner{\mathrm{d}{#1}}}
\newcommand{\modulo}[1]{{\left|#1\right|}}
\newcommand{\norma}[1]{{\left\|#1\right\|}}
\newcommand{\CC}[1]{{\mathbf{C^#1}}}

\begin{document}

\date{}

\title{Optimal control problems for a parabolic system modeling glioma
  therapy}

\author{Mauro Garavello\thanks{E-mail:
    \texttt{mauro.garavello@unimib.it}.
  }\\
  Department of Mathematics and its Applications,\\
  University of Milano Bicocca,\\
  via R. Cozzi 55,\\
  20125 Milano (Italy). \\
  \and Elena Rossi\thanks{E-mail: \texttt{elena.rossi13@unimore.it}.}\\
  Department of Sciences and Methods for Engineering,\\
  University of Modena and Reggio Emilia,\\
  Via Amendola 2, Pad.~Morselli \\
  42122 Reggio Emilia (Italy)}

\maketitle

\begin{abstract}
  In this paper we consider optimal control problems for a
  parabolic system modeling a therapy, based on oncolytic viruses, for
  the glioma brain cancer.  Using several techniques typical of
  functional analysis, we prove the global in time well posedness of
  the control model, the existence of optimal controls for specific
  objective functionals, which are natural for cancer therapies, and we derive necessary conditions for
  optimality.
\end{abstract}



\section{Introduction}\label{se:introduction}
In this paper we consider optimal control problems for a $3\times3$
system of parabolic partial differential equations modeling a therapy
in the case of a brain cancer, the glioma one, based on the infusion
of oncolytic viruses. They are genetically modified viruses able to
infect cancer cells and to replicate inside them, but they are not
harmful for healthy cells. With this mechanism, they eventually kill
mainly cancer cells. Moreover, when an infected cell dies, it releases
many copies of the viruses, which then spread to infect neighboring
tumor cells.  The main obstacle in the use of oncolytic viruses
consists in the fact that the innate immune system
recognizes the cells infected by the virus and destroys them before
the virus multiply. In this paper, as in~\cite{FriedmanTao}
and~\cite{WKBN}, we are neglecting this aspect.

There is a huge mathematical literature for cancer modeling based on
differential equations; see~\cite{MR2493727, MR3059566,  MR1952568,
  MR4237849} and the reference therein for a detailed description.
This is essentially due to the large variety of diseases commonly
named under the word \textsl{cancer}.  Each tumor has some specific
peculiarities and dynamics; hence it requires an ad-hoc model for a
precise mathematical description.  Also therapies vary
accordingly. For example, they include chemotherapy, radiotherapy, stem
cell transplant, surgery and can be dosed also combined together. This
justifies the large number of mathematical papers dealing with
this subject.
In particular, considering Glioma type cancer,
we can distinguish the various models through different categories: based on ODEs~\cite{MR3562913, MR3721853, MR3767838} or on PDEs~\cite{CherfilsGattiMiranvilleGuillevin,  ContiGattiMiranville, FriedmanTao, FriedmanPDEmodel, WKBN},
focusing on controlling aspects~\cite{ContiGattiMiranville, MR3562913, MR3721853, MR3767838}, on therapy calibration~\cite{Biesecker, FriedmanTao, FriedmanPDEmodel, WKBN} or on asymptotic behaviour of
solutions~\cite{CherfilsGattiMiranvilleGuillevin, FriedmanTao}.

The main results {of this paper} are the well posedness of the parabolic $3 \times 3$
control system, the existence of optimal controls, and first order
necessary conditions for optimality. They are
obtained through several techniques typical of functional analysis.
In particular, different from other papers in the literature,
see for instance~\cite{CherfilsGattiMiranvilleGuillevin, ContiGattiMiranville}, we use the Banach fixed point theorem to prove local
in time existence and uniqueness of solution for the parabolic system.
Moreover, a combination of a-priori estimates and maximum principles
for scalar equations permits to extend the solution to arbitrary time
intervals obtaining global in time well posedness.  Gronwall
inequality is used to prove the Lipschitz continuity, in the
$\LL2$ topology, of the solution with respect to the controls.
Existence of optimal controls is deduced using the direct
method in the calculus of variation; see for example~\cite{MR2361288}.
Here continuous embedding theorems and Ascoli-Arzelà Theorem play an
important role in the weak and strong convergence of quasi-optimal
solutions.
Finally, necessary conditions are obtained through the
derivative of the input-output map and the adjoint
system. 

The main novelty of the paper consists in the study of
a nonlinear system of parabolic partial differential equations with an
open loop control function; see~\cite{zbMATH05150528, zbMATH01394147,
  zbMATH00733958, zbMATH05703572} and the references therein for
control problems for partial differential equations.  Here we derive
necessary and sufficient conditions for optimal controls.

The paper is organized as follows. \Cref{se:def} introduces the
mathematical model and the definition of solution.
In~\Cref{sec:local-global-existence} we prove the existence and
uniqueness of solution both local and global in time.
In~\Cref{sec:dep-sol-control} we study the Lipschitz continuous
dependence of the solution with respect to the control function, while
in~\Cref{sec:optimal-control} we deduce the existence of optimal
controls for some objective functionals, natural for cancer therapies.
\Cref{sec:necessary-conditions} deals with necessary conditions for
optimality.  Finally~\Cref{sec:base} contains classical results about
well posedness of scalar parabolic equations, used in~\Cref{sec:local-global-existence}. It is mainly intended to ease the
readability of the paper.

%
%
%
%
\section{Basic definitions and notations}\label{se:def}

In this paper we consider control problems for the system of partial
differential equations
\begin{equation}
  \label{eq:PDE-model-3x3}
  \left\{
    \begin{array}{l}
      \pt \rho_1 = \Delta \rho_1 + \left(\alpha - \delta_1\right) \rho_1
      - \beta \rho_1 v
      \\
      \pt \rho_2 = \Delta \rho_2 + \beta \rho_1 v - \delta_2 \rho_2
      \\
      \pt v = \Delta v + b \delta_2 \rho_2 - B \rho_1 v - \delta_v v + u,
    \end{array}
  \right.
\end{equation}
where $t \ge 0$ is the time, $x \in \Omega$ is the spatial variable,
$\Omega \subseteq \R^N$ is an open, bounded, and connected set with
smooth boundary denoted by $\partial \Omega$, and $N \in \N$,
$N \ge 2$ (typically $N=2$ or $N=3$ in applications).  Moreover
$\rho_1, \rho_2:(0,+\infty) \times \Omega \to \R$ describe the density
respectively of uninfected cancer cells and of infected cancer cells,
while $v:(0,+\infty) \times \Omega \to \R$ represents the density of
the injected virus.  The map {$u=u(t,x)$} is the control
function modeling the velocity of the virus infusion.  Finally
$\alpha$, $\beta$, $\delta_1$, $\delta_2$, $\delta_v$, $b$, $B$ are
fixed positive constants.
In the paper we consider controls depending also on
the spatial variable, although, as detailed in~\cite{DOI},
in real situations the viral therapy is administered intravenously,
so that a control depending only on time should be more
realistic.

We augment the system~\eqref{eq:PDE-model-3x3} with the initial
conditions
\begin{equation}
  \label{eq:initial-condition}
  \left\{
    \begin{array}{l}
      \rho_1(0, x) = \rho_{1,o}(x)
      \\
      \rho_2(0, x) = \rho_{2,o}(x)
      \\
      v(0, x) = v_o(x),
    \end{array}
  \right.
\end{equation}
where $\rho_{1,o}, \rho_{2,o}, v_{o} \in \LL2\left(\Omega\right)$, and
with homogeneous Neumann boundary conditions
\begin{equation}
  \label{eq:boundary-condition}
  \left\{
    \begin{array}{l}
      \partial_\nu \rho_1(t, \xi) = 0
      \\
      \partial_\nu \rho_2(t, \xi) = 0
      \\
      \partial_\nu v(t, \xi) = 0
    \end{array}
  \right.
\end{equation}
for $\xi \in \partial \Omega$, where the symbol $\partial_\nu$ denotes
the inner normal derivative.

Throughout the paper, we deal with the following concept of weak
solution for
system~\eqref{eq:PDE-model-3x3}-\eqref{eq:initial-condition}-\eqref{eq:boundary-condition}.
\begin{definition}
  \label{def:sol}
  Given $T > 0$, the triple $(\rho_1, \rho_2, v)$ is a \emph{solution}
  to the initial-boundary value
  problem~\eqref{eq:PDE-model-3x3}-\eqref{eq:initial-condition}-\eqref{eq:boundary-condition}
  on the time interval $[0,T]$ if
  \begin{enumerate}
  \item $\rho_1, \rho_2, v \in \LL\infty((0,T) \times \Omega; \R)$;
  \item $\rho_1, \rho_2, v \in \LL2([0,T]; \HH1(\Omega))$;
  \item
    $\dot \rho_1, \dot \rho_2, \dot v \in \LL2([0,T];
    \HH1(\Omega)^*)$;
  \item $\rho_1(0,x) = \rho_{1,o}(x)$, $\rho_2(0,x) = \rho_{2,o}(x)$,
    $v(0,x) = v_o(x)$ in $\LL2(\Omega)$;
  \item for a.e.~$t \in [0,T]$ and for any
    $w_1, w_2, w_3 \in \HH1(\Omega)$
    \begin{align*}
      \langle \dot \rho_1(t), w_1\rangle =
      & - \!\int_\Omega \nabla \rho_1(t,x) \cdot \nabla w_1(x) \dd{x}
        + (\alpha - \delta_1)\! \int_\Omega \rho_1(t,x) \,  w_1(x) \dd{x}
      \\
      & - \beta \int_\Omega \rho_1(t,x) \, v(t,x) \, w_1(x) \dd{x},
      \\
      \langle \dot \rho_2(t), w_2\rangle =
      & - \int_\Omega \nabla \rho_2(t,x) \cdot \nabla w_2(x) \dd{x}
        - \delta_2 \int_\Omega \rho_2(t,x) \, w_2(x) \dd{x}
      \\
      & + \beta \int_\Omega \rho_1(t,x) \, v(t,x) \, w_2(x) \dd{x},
      \\
      \langle \dot v(t), w_3\rangle =
      & - \int_\Omega \nabla v(t,x) \cdot \nabla w_3(x) \dd{x}
        + b \, \delta_2 \int_\Omega \rho_2(t,x) \, w_3(x) \dd{x}
      \\
      & - B \int_\Omega \rho_1(t,x) \, v(t,x) \, w_3(x) \dd{x}
        - \delta_v \int_\Omega v(t,x) \, w_3(x) \dd{x}
      \\
      & {+ \int_\Omega u(t,x) \, w_3(x) \dd{x}}.
    \end{align*}
  \end{enumerate}
\end{definition}

\begin{remark}
  \label{rmk:sol-C0}
  Note that assumptions 2 and 3 of~\Cref{def:sol} imply that the
  functions $\rho_1$, $\rho_2$, and $v$ belong to the space
  $\CC0\left([0, T]; \LL2\left(\Omega\right)\right)$;
  see~\cite[Theorem~7.104]{Salsa}.  This justifies the condition 4
  of~\Cref{def:sol}.
\end{remark}

\subsection{Model justification}

In~\cite{MR3562913, MR3721853, MR3767838} the authors proposed a
mathematical model for the therapy of glioma based on oncolytic
viruses infusion.  Oncolytic viruses are genetically altered viruses
able to infect cancer cells but not normal ones. They reproduce in
cancer cells and eventually kill them, and when an infected cell dies,
many new viruses are released and spread out. The model
in~\cite{MR3562913, MR3721853, MR3767838} is given by
the following system of nonlinear ordinary differential equations
\begin{equation}
  \label{eq:complete-model-5x5}
  \left\{
    \begin{array}{l}
      \dot x = \alpha x - \beta x v - \delta_x x
      \\
      \dot y = \beta x v - \xi y \frac{T}{K + T} - \delta_y y
      \\
      \dot M = A + s y M - \delta_M M
      \\
      \dot T = \frac{\eta}{1 + u_2} M - \omega y \frac{T}{K + T} - \delta_T T
      \\
      \dot v = b \delta_y y - \rho x v - \delta_v v + u_1,
    \end{array}
  \right.
\end{equation}
where the unknowns $x$, $y$, $M$, $T$, $v$ represent respectively the
density of uninfected cancer stem cells, the density of infected
cancer cells, the density of the macrophages, the concentration of
TNF-$\alpha$ inhibitors, and the density of the virus. The control
functions $u_1 = u_1(t)$ and $u_2 = u_2(t) $ denote respectively the
amount of virus and of TNF-$\alpha$ inhibitor that is injected at time
$t$.  The descriptions and realistic values of the various parameters
appearing in system~\eqref{eq:complete-model-5x5} can be found
in~\cite[Table~2]{MR3767838}.  One can also note in~\cite{MR3767838}
that the dynamics of the unknowns $M$ and $T$ is almost static around
the values $M \sim 0.1 \frac{g}{cm^3}$ and
$T \sim 5\times
10^{-6}\frac{g}{cm^3}$. Hence~\eqref{eq:complete-model-5x5} can be
approximated by the $3\times3$ system
\begin{equation}
  \label{eq:reduced-model-3x3}
  \left\{
    \begin{array}{l}
      \dot x = \alpha x - \beta x v - \delta_x x
      \\
      \dot y = \beta x v - \tilde \xi y - \delta_y y
      \\
      \dot v = b \delta_y y - \rho x v - \delta_v v + u_1.
    \end{array}
  \right.
\end{equation}
Model~\eqref{eq:PDE-model-3x3} is the natural generalization
of~\eqref{eq:reduced-model-3x3} once we allow the densities of cancer
cells and of the virus to depend also on the spatial coordinate.

\section{Local and global existence}
\label{sec:local-global-existence}
In this section we prove both the local and global well posedness for
system~\eqref{eq:PDE-model-3x3}. The local in time result is proved
using a fixed point technique, while a-priori estimates permit to
extend the solution to arbitrary time intervals.
In the following we use the notation $\Omega_T = (0, T) \times \Omega$.

\begin{theorem}
  \label{thm:local_wellposedness}
  Assume $\alpha$, $\beta$, $\delta_1$, $\delta_2$, $\delta_v$, $b$,
  $B$, and $U$ fixed positive constants.  Let $\Omega \subseteq \R^N$
  be an open, connected, and bounded domain, with smooth boundary
  $\partial \Omega$. Fix
  $\rho_{1,o}, \rho_{2,o}, v_{o} \in \LL\infty\left(\Omega; \R_+\right)$ and
  {$u \in \LL\infty(\R \times \Omega; \R_+)$, with
    $\norm{u}_{\LL\infty(\R \times \Omega)} \le U$}. There exist
  $T > 0$ and a unique solution $\left(\rho_1, \rho_2, v\right)$
  to~\eqref{eq:PDE-model-3x3}-\eqref{eq:initial-condition}-\eqref{eq:boundary-condition}
  on the time interval $[0, T]$, in the sense of~\Cref{def:sol}.
  Moreover, for a.e. $t \in [0, T]$ and $x \in \Omega$,
  \begin{equation}
    \label{eq:positivity-solution}
    \rho_1(t, x) \ge 0, \qquad
    \rho_2(t, x) \ge 0, \qquad
    v(t, x) \ge 0.
  \end{equation}
  Finally, if moreover
  $\rho_{1,o}, \rho_{2,o}, v_{o} \in \HH1 \left(\Omega; \R_+\right)$, then
  \begin{equation*}
    \rho_1, \rho_2, v \in \LL2\left(0, T; \HH2\left(\Omega\right)\right).
  \end{equation*}
\end{theorem}

\begin{proof}
  Define
  \begin{equation}
    \label{eq:M}
    M = 2 \max \left\{\norm{\rho_{1,o}}_{\LL\infty\left(\Omega\right)},
      \norm{\rho_{2,o}}_{\LL\infty\left(\Omega\right)},
      \norm{v_{o}}_{\LL\infty\left(\Omega\right)}\right\} + 1.
  \end{equation}
  Fix $T > 0$ such that
  \begin{equation}
    \label{eq:T-small}
    \begin{aligned}
      T < \ & \min\left\{\frac{1}{\alpha + \delta_1+ \frac{3M-1}{2} \,
          \beta} \ln \left(\frac{M+1}{M - 1}\right),\,
        \frac{1}{\delta_2}\ln\left(\frac{M}{M-1}\right), \right.
      \\
      & \frac{1}{\delta_2}\ln\left(1 + \frac{2 \, \delta_2}{\beta \,
          (3M-1)^2}\right),\, \frac{2}{(3M-1)B + 2 \,
        \delta_v}\ln\left(\frac{M}{M-1}\right),\,
      \\
      & \frac{2}{(3M-1)B + 2 \, \delta_v}\ln\left(1 +
        \frac{\delta_v}{(3M-1)b \, \delta_2 + 2 \, U}\right),
      \\
      & \frac{2}{2 \,\alpha + 2 \,\delta_1+ (3M-1) \, \beta} \ln
      \left(\frac32\right), \, \frac{1} {9 \beta^2 \left(3M -
          1\right)^2}, \, \frac{1}{\delta_2}\ln\left(\frac32\right)
      \\
      & \left.  \frac{2}{2 \, \delta_v + (3 \, M-1) B }
        \ln\left(\frac32\right), \, \frac{1}{36 \, b^2 \, \delta_2^2},
        \, \frac{1} {9 B^2 \left(3M - 1\right)^2} \right\}
    \end{aligned}
  \end{equation}
  Consider the Banach spaces
  \begin{align*}
    X_1
    & = \left\{\zeta \in \CC0\left([0, T]; \LL2\left(\Omega\right)\right):
      \sup_{t \in [0, T]} \norm{\zeta(t) - \rho_{1, o}}_{\LL\infty\left(\Omega\right)}
      \le M
      \right\}
    \\
    X_2
    & = \left\{\zeta \in \CC0\left([0, T]; \LL2\left(\Omega\right)\right):
      \sup_{t \in [0, T]} \norm{\zeta(t) - \rho_{2, o}}_{\LL\infty\left(\Omega\right)}
      \le M
      \right\}
    \\
    X_3
    & = \left\{\zeta \in \CC0\left([0, T]; \LL2\left(\Omega\right)\right):
      \sup_{t \in [0, T]} \norm{\zeta(t) - v_{o}}_{\LL\infty\left(\Omega\right)}
      \le M
      \right\}
  \end{align*}
  endowed with the norm
  \begin{equation*}
    \norm{\zeta}_{X_i} = \sup_{t \in [0, T]}
    \norm{\zeta(t)}_{\LL2\left(\Omega\right)}
  \end{equation*}
  for $i \in \left\{1, 2, 3\right\}$.

  Define $X = X_1 \times X_2 \times X_3$ with the norm
  $\norm{\left(\zeta_1, \zeta_2, \zeta_3\right)}_X = \sum_{i=1}^3
  \norm{\zeta_i}_{X_i}$ and the map
  \begin{equation*}
    \mathcal T: X \longrightarrow X
  \end{equation*}
  such that, for every $\left(r_1, r_2, w\right) \in X$,
  $\mathcal T \left(r_1, r_2, w\right) = \left(\rho_1, \rho_2,
    v\right)$ is the unique weak solution to the decoupled system
  \begin{equation*}
    \left\{
      \begin{array}{l}
        \pt \rho_1 = \Delta \rho_1 + \left(\alpha - \delta_1 - \beta \, w \right) \rho_1
        \\
        \pt \rho_2 = \Delta \rho_2 - \delta_2 \rho_2 + \beta r_1 w
        \\
        \pt v = \Delta v + b \delta_2 r_2  + u -(B r_1 + \delta_v)  v
      \end{array}
    \right.
  \end{equation*}
  with initial data $\left(\rho_{1,o}, \, \rho_{2,o}, \, v_o\right)$
  and Neumann homogeneous boundary conditions.  Such solution exists
  by~\Cref{thm:parBase}, since the functions
  $r_1 w, r_2, u \in \LL2\left(0, T;
    \HH1\left(\Omega\right)^*\right)$.
  Indeed, to comply with~\Cref{thm:parBase}:
  \begin{itemize}
  \item
    {$u\in\LL\infty((0,T)\times \Omega ) \subseteq \LL2((0,T)
      \times \Omega)$}
  \item $r_2 \in \LL2(\Omega_T)$, being in $X_2$
  \item $r_1, \, w \in \LL2(\Omega_T)$, being respectively in $X_1$
    and $X_3$.
  \item $w, r_1$ are needed to be in $\LL\infty(\Omega_T)$ and this is
    true since, by~\eqref{eq:M}, for instance
    \begin{align*}
      \modulo{r_1(t,x)}
      & \leq \modulo{r_1(t,x) - \rho_{1,o}(x)} + \modulo{\rho_{1,o}(x)}
        \leq M + \norma{\rho_{1,o}}_{\LL\infty(\Omega)}
      \\
      & \leq M + \frac{M-1}{2}.
    \end{align*}
  \end{itemize}

  We observe that if $r_1 \in X_1$, $r_2 \in X_2$ and $w \in X_3$,
  then
  \begin{equation}
    \label{eq:1}
    \norm{r_1}_{\LL\infty (\Omega)}, \, \norma{r_2}_{\LL\infty (\Omega)},
    \norm{w}_{\LL\infty (\Omega)} \leq \frac{3\, M-1}{2}.
  \end{equation}

  \medskip

  We need to show that $\mathcal T$ is well defined, in the sense that
  $\left(\rho_1, \rho_2, v\right) \in X$.  First note that $\rho_1$,
  $\rho_2$, and $v$ belong to
  $\HH1\left(0, T; \HH1\left(\Omega\right),
    \HH1\left(\Omega\right)^*\right)$, see~\eqref{eq:hilbert-space},
  and so, by~\cite[Theorem~7.104]{Salsa}, to the space
  $\CC0\left([0, T]; \LL2\left(\Omega\right)\right)$.

  Consider first the case of $\rho_1$.  By~\Cref{prop:a-priori-growth}
  we deduce that
  \begin{displaymath}
    0 \leq \rho_1(t,x) \leq
    \norm{\rho_{1, o}}_{\LL\infty (\Omega)} e^{\left(\alpha + \delta_1 + \frac{3M-1}{2}\, \beta\right)  t}.
  \end{displaymath}
  Therefore,
  \begin{align*}
    \modulo{\rho_1(t,x) - \rho_{1,o}(x)} \le \
    & \abs{\rho_1(t,x)} + \abs{\rho_{1,o}(x)}
    \\
    = \
    & \rho_1(t,x) + \rho_{1,o}(x)
    \\
    \leq \
    & \frac{M-1}{2}\left( 1 + e^{\left(\alpha + \delta_1 + \frac{3M-1}{2}\, \beta\right)  t}\right).
  \end{align*}
  By~\eqref{eq:T-small},
  \begin{displaymath}
    \sup_{t \in [0, T]}\norm{\rho_1(t) - \rho_{1, o}}_{\LL\infty (\Omega)}
    \leq
    \frac{M-1}{2}\left( 1 + e^{\left(\alpha + \delta_1 + \frac{3M-1}{2}\, \beta\right)  T}\right)
    < M,
  \end{displaymath}
  proving that $\rho_1 \in X_1$.

  Pass now to $\rho_2$. \Cref{prop:a-priori-growth} yields
  \begin{displaymath}
    0 \leq \rho_2(t,x) \leq
    \left( \norma{\rho_{2,o}}_{\LL\infty(\Omega)} + \frac{\beta \, \norm{r_1\,w}_{\LL\infty(\Omega_t)}}{\delta_2}\right)
    e^{\delta_2 \, t} - \frac{\beta \, \norm{r_1\,w}_{\LL\infty(\Omega_t)}}{\delta_2}.
  \end{displaymath}
  By~\eqref{eq:1}, we deduce that
  \begin{displaymath}
    \norm{r_1 \, w}_{\LL\infty (\Omega)} \leq \frac{(3M-1)^2}{4}.
  \end{displaymath}
  Thus,
  \begin{align*}
    \abs{\rho_2(t,x) - \rho_{2,o}(x)}
    \leq \
    & \norma{\rho_{2,o}}_{\LL\infty(\Omega)} \left( 1 + e^{\delta_2 \, t}\right)
      +
      \frac{\beta \, \norm{r_1\,w}_{\LL\infty(\Omega_t)}}{\delta_2} \left(
      e^{\delta_2 \, t} -1
      \right)
    \\
    \leq \
    & \frac{M-1}{2}\left( 1 + e^{\delta_2 \,t}\right)
      + \frac{\beta}{\delta_2} \, \frac{(3M-1)^2}{4} \left( e^{\delta_2 \, t } -1\right).
  \end{align*}
  By~\eqref{eq:T-small} we obtain
  \begin{align*}
    \sup_{t \in [0, T]}\norm{\rho_2(t) - \rho_{2, o}}_{\LL\infty (\Omega)}
    \leq \
    & \frac{M-1}{2}\left( 1 + e^{\delta_2 \,T}\right)
      + \frac{\beta}{\delta_2} \, \frac{(3M-1)^2}{4} \left( e^{\delta_2 \, T } -1\right)
    \\
    \leq \
    & \left(M - \frac{1}{2} \right)+ \frac12 = M,
  \end{align*}
  proving that $\rho_2 \in X_2$.

  Consider now $v$.  \Cref{prop:a-priori-growth} yields
  \begin{align*}
    0 \leq
    v(t,x) \leq\
    &
      \left( \norma{v_{o}}_{\LL\infty(\Omega)} + \frac{b \, \delta_2 \, \norm{r_2}_{\LL\infty(\Omega_t)} + U}{B \, \norm{r_1}_{\LL\infty(\Omega_t)} + \delta_v}\right)
      e^{\left(B \, \norm{r_1}_{\LL\infty(\Omega_t)} + \delta_v \right)\, t}
    \\
    & -
      \frac{b \, \delta_2 \, \norm{r_2}_{\LL\infty(\Omega_t)} + U}{B \, \norm{r_1}_{\LL\infty(\Omega_t)} + \delta_v}.
  \end{align*}

  Exploiting~\eqref{eq:M} and~\eqref{eq:1}, we obtain
  \begin{align*}
    0 \leq
    v(t,x) \leq\
    &
      \norma{v_o}_{\LL\infty (\Omega)} \,  e^{\left(B \, \norm{r_1}_{\LL\infty(\Omega_t)} + \delta_v \right)\, t}
    \\
    & +
      \frac{b \, \delta_2 \, \norm{r_2}_{\LL\infty(\Omega_t)} + U}{B \, \norm{r_1}_{\LL\infty(\Omega_t)} + \delta_v}
      \left( e^{\left(B \, \norm{r_1}_{\LL\infty(\Omega_t)} + \delta_v \right)\, t}-1\right)
    \\
    \leq \
    &
      \frac{M-1}{2} +
      \frac{b \, \delta_2 \, (3M-1) + 2 \, U}{2 \, \delta_v}
      \left( e^{\left(B \,\frac{3M-1}{2}+ \delta_v \right)\, t} -1 \right).
  \end{align*}
  Thus
  \begin{align*}
    \abs{v(t,x) - v_{o}(x)}
    \leq \
    & \frac{M-1}{2}\left( 1 + e^{\left(B \,\frac{3M-1}{2}+ \delta_v \right)\, t}\right)
    \\
    & +  \frac{b \, \delta_2 \, (3M-1) + 2 \, U}{2 \, \delta_v}
      \left( e^{\left(B \,\frac{3M-1}{2}+ \delta_v \right)\, t} -1\right).
  \end{align*}
  By~\eqref{eq:T-small} we obtain
  \begin{align*}
    \sup_{t \in [0, T]}\norm{v(t) - v_o}_{\LL\infty(\Omega)}
    \leq \
    & \frac{M-1}{2}\left( 1 + e^{\left(B \,\frac{3M-1}{2}+ \delta_v \right)\, T}\right)
    \\
    & +  \frac{b \, \delta_2 \, (3M-1) + 2 \, U}{2 \, \delta_v}
      \left( e^{\left(B \,\frac{3M-1}{2}+ \delta_v \right)\, T} -1\right)
    \\
    \leq \
    & \left(M - \frac{1}{2} \right)+ \frac12 = M,
  \end{align*}
  proving that $v\in X_3$.

  \medskip

  Fix $\left(\bar r_1, \bar r_2, \bar w\right) \in X$ and
  $\left(\tilde r_1, \tilde r_2, \tilde w\right) \in X$.  Define
  $\left(\bar \rho_1, \bar \rho_2, \bar v\right) = \mathcal T
  \left(\bar r_1, \bar r_2, \bar w\right)$ and
  $\left(\tilde \rho_1, \tilde \rho_2, \tilde v\right) = \mathcal T
  \left(\tilde r_1, \tilde r_2, \tilde w\right)$.

  Note that
  \begin{equation*}
    \left\{
      \begin{array}{l}
        \partial_t \left(\bar \rho_1 - \tilde \rho_1\right) = \Delta
        \left(\bar \rho_1 - \tilde \rho_1\right)
        + \left(\alpha - \delta_1 - \beta \bar w\right)
        \left(\bar \rho_1 - \tilde \rho_1\right)
        + \beta \left(\tilde w - \bar w\right) \tilde \rho_1
        \\
        \left(\bar \rho_1 - \tilde \rho_1\right)(0, x) = 0
        \\
        \partial_\nu \left(\bar \rho_1 - \tilde \rho_1\right)(t, \xi) = 0.
      \end{array}
    \right.
  \end{equation*}
  \Cref{thm:parBase} implies that, for $t \in [0, T]$,
  \begin{align*}
    & \norma{\bar \rho_1(t) - \tilde \rho_1(t)}_{\LL2\left(\Omega\right)}^2
    \\
    \le
    & e^{2 \norma{\alpha - \delta_1 - \beta \bar w}_{\LL\infty\left(\Omega_T\right)}t}
      \int_0^t \norma{\beta \left(\tilde w(s) - \bar w(s)\right) \tilde \rho_1(s)}_{\HH1\left(\Omega\right)^*}^2 \dd s
    \\
    \le
    & e^{2 \left(\alpha + \delta_1 + \beta \frac{3M-1}{2}\right)T}
      \int_0^t \norma{\beta \left(\tilde w(s) - \bar w(s)\right) \tilde \rho_1(s)}_{\LL2\left(\Omega\right)}^2 \dd s
    \\
    \le
    & \beta^2 e^{2 \left(\alpha + \delta_1 + \beta \frac{3M-1}{2}\right)T}
      \frac{\left(3M-1\right)^2}{4}
      \int_0^t \norma{\tilde w(s) - \bar w(s)}_{\LL2\left(\Omega\right)}^2 \dd s
    \\
    \le
    & \beta^2 e^{2 \left(\alpha + \delta_1 + \beta \frac{3M-1}{2}\right)T}
      \frac{\left(3M-1\right)^2}{4}
      \sup_{s \in [0, T]} \norma{\tilde w(s) - \bar w(s)}_{\LL2\left(\Omega\right)}^2 T,
  \end{align*}
  where we used the fact that $\tilde \rho_1 \in X_1$.  Therefore,
  using~\eqref{eq:T-small},
  \begin{align*}
    \norma{\bar \rho_1 - \tilde \rho_1}_{X_1}
    & \le \beta e^{\left(\alpha + \delta_1 + \beta \frac{3M-1}{2}\right)T}
      \frac{3M-1}{2}  \sqrt{T}
      \norma{\tilde w - \bar w}_{X_{3}}
    \\
    & \le \frac{1}{4} \norma{\tilde w - \bar w}_{X_{3}}.
  \end{align*}

  Proceed analogously for the other two equations.  Indeed, since
  \begin{equation*}
    \left\{
      \begin{array}{l}
        \partial_t \left(\bar \rho_2 - \tilde \rho_2\right) = \Delta
        \left(\bar \rho_2 - \tilde \rho_2\right)
        - \delta_2 \left(\bar \rho_2 - \tilde \rho_2\right)
        -\beta \, \bar r_1 \, \bar w
        + \beta \, \tilde r_1 \, \tilde w
        \\
        \left(\bar \rho_2 - \tilde \rho_2\right)(0, x) = 0
        \\
        \partial_\nu \left(\bar \rho_2 - \tilde \rho_2\right)(t, \xi) = 0,
      \end{array}
    \right.
  \end{equation*}
  \Cref{thm:parBase} implies that, for $t \in [0, T]$,
  \begin{displaymath}
    \norm{\bar \rho_2(t) - \tilde \rho_2(t)}_{\LL2\left(\Omega\right)}^2
    \le
    e^{2 \, \delta_2 \, t}
    \int_0^t \beta^2
    \norma{
      \left(\bar r_1 \, \bar w - \tilde r_1 \, \tilde w\right) (s)
    }_{H^1 (\Omega)^*} \dd{s}.
  \end{displaymath}
  Observe that for $s \in [0, T]$, we have
  \begin{align*}
    & \norm{\left(\bar r_1 \, \bar w
      - \tilde r_1 \,  \tilde w\right)(s)}_{\HH1\left(\Omega\right)^*}
    \\
    \le
    & \norm{\bar r_1(s) \left(\bar w(s)
      - \tilde w(s)\right)}_{\HH1\left(\Omega\right)^*}
      +\norm{\left(\bar r_1(s)
      - \tilde r_1(s)\right) \tilde w(s)}_{\HH1\left(\Omega\right)^*}
    \\
    \le
    & \frac{3\, M-1}{2} \left( \norm{\left(\bar w - \tilde w\right) (s)}_{\LL2\left(\Omega\right)}
      + \norm{\left(\bar r_1 - \tilde r_1\right) (s)}_{\LL2\left(\Omega\right)}\right),
  \end{align*}
  since $\bar r_1 \in X_1$ and $\tilde w \in X_3$ and we
  exploited~\eqref{eq:1}.  Therefore, using~\eqref{eq:T-small},
  \begin{align*}
    \norm{\bar \rho_2 - \tilde \rho_2}_{X_2}
    \le \
    & e^{\delta_2 \, T} \, \beta \, \frac{3 \, M -1}{2} \, \sqrt{T}
      \left( \norm{\bar w - \tilde w}_{X_3}
      + \norm{\bar r_1 - \tilde r_1}_{X_1}\right)
    \\
    \leq \
    & \frac{1}{4} \left( \norm{\bar w - \tilde w}_{X_3}
      + \norm{\bar r_1 - \tilde r_1}_{X_1}\right).
  \end{align*}
  Lastly we have
  \begin{equation*}
    \left\{
      \begin{array}{l}
        \partial_t \left(\bar v - \tilde v\right) = \Delta
        \left(\bar v - \tilde v\right)
        - \left( \delta_v  + B \, \bar r_1 \right)
        \left(\bar v - \tilde v\right)
        -B \, \left(\bar r_1 - \tilde r_1 \right) \tilde v
        \\
        \qquad \qquad \quad\, + b \, \delta_2 \left(\bar r_2 - \tilde r_2\right)
        \\
        \left(\bar v - \tilde v\right)(0, x) = 0
        \\
        \partial_\nu \left(\bar v - \tilde v\right)(t, \xi) = 0.
      \end{array}
    \right.
  \end{equation*}
  Again, \Cref{thm:parBase} implies that, for $t \in [0, T]$,
  \begin{align*}
    &\norma{\bar v(t) - \tilde v(t)}_{\LL2\left(\Omega\right)}^2
    \\
    \le \
    & e^{2 \norma{\delta_v + B \, \bar r_1}_{\LL\infty\left(\Omega_T\right)}t}
      \int_0^t \norma{\beta \, \delta_2 \left(\bar r_2(s) - \tilde r_2(s)\right)
      - B \, \left(\bar r_1(s) - \tilde r_1(s)\right) }_{\LL2\left(\Omega\right)}^2
      \dd s
    \\
    \le\
    & e^{2 \left(\delta_v + B \, \frac{3M-1}{2}\right)T}
      \int_0^t \norma{\beta \, \delta_2 \left(\bar r_2(s) - \tilde r_2(s)\right)
      - B \, \left(\bar r_1(s) - \tilde r_1(s)\right) \tilde v(s)
      }_{\LL2\left(\Omega\right)}^2
      \dd s
    \\
    \le \
    & e^{2 \left(\delta_v + B \, \frac{3M-1}{2}\right)T}
      T \, \Biggl(\beta \, \delta_2
      \sup_{s \in [0,T]} \norma{\bar r_2(s) - \tilde r_2(s)}_{\LL2\left(\Omega\right)}
    \\
    & \qquad\left.+ B \, \frac{3 \, M -1}{2}
      \sup_{s \in [0,T]} \norma{\bar r_1(s) - \tilde r_1(s)}_{\LL2\left(\Omega\right)}
      \right)^2,
  \end{align*}
  where we used the fact that $\tilde v \in X_3$.  Hence,
  using~\eqref{eq:T-small},
  \begin{align*}
    &
      \norma{\bar v - \tilde v}_{X_3}
    \\
    \le
    &  e^{\left(\delta_v + B \, \frac{3M-1}{2}\right)T} \,
      \sqrt{T} \,
      \max\left\{ b \, \delta_2,  B \, \frac{3 \, M-1}{2}\right\}
      \left(\norma{\bar r_1 - \tilde r_1}_{X_1}
      + \norma{\bar r_2 - \tilde r_2}_{X_2}\right)
    \\
    \le
    & \frac{1}{4}  \left(\norma{\bar r_1 - \tilde r_1}_{X_1}
      + \norma{\bar r_2 - \tilde r_2}_{X_2}\right).
  \end{align*}

  Therefore, for $t \in [0, T]$,
  \begin{align*}
    &
      \norm{\bar \rho_1 - \tilde \rho_1}_{X_1}
      + \norm{\bar \rho_2 - \tilde \rho_2}_{X_2}
      + \norm{\bar v - \tilde v}_{X_3}
    \\
    \le\
    &\frac14
      \left(
      2 \, \norma{\bar r_1 - \tilde r_1}_{X_1}
      + \norma{\bar r_2 - \tilde r_2}_{X_2}
      +2 \, \norma{\bar w - \tilde w}_{X_3}
      \right)
    \\
    \le \
    & \frac12 \left(
      \norma{\bar r_1 - \tilde r_1}_{X_1}
      + \norma{\bar r_2 - \tilde r_2}_{X_2}
      + \norma{\bar w - \tilde w}_{X_3}
      \right),
  \end{align*}
  proving that $\mathcal{T}$ is a contraction. Banach Fixed Point
  Theorem implies that the map $\mathcal{T}$ admits a unique fixed
  point in $X$, thus ensuring the existence of solutions
  to~\eqref{eq:PDE-model-3x3} on the time interval $[0,T]$. Observe
  that the solution $(\rho_1, \rho_2, v)$ preserves the positivity of
  the initial data $(\rho_{1,o},\rho_{2,o}, v_o)$.

  If the initial data are in
  $\HH1  (\Omega; \R_+)$, then, due to
  \Cref{prop:regularity}, each
  component of $\mathcal{T}$ maps
  $\LL2 \left(0,T; \HH2 (\Omega)\right)$ into itself.
  Therefore, repeating the same
  argument as above yields
  $\rho_2, \, \rho_2, \, v \in \LL2\left(0,T; \HH2(\Omega)\right)$.
\end{proof}

The next result deals with the global existence of solutions.

\begin{theorem}
  \label{thm:global-existence}
  Assume $\alpha$, $\beta$, $\delta_1$, $\delta_2$, $\delta_v$, $b$,
  $B$, and $U$ fixed positive constants.  Let $\Omega \subseteq \R^N$
  be an open, connected, and bounded domain, with smooth boundary
  $\partial \Omega$. Fix
  $\rho_{1,o}, \rho_{2,o}, v_{o} \in
    \LL\infty\left(\Omega; \R_+\right)$ and
  {$u \in \LL\infty(\R \times \Omega; \R_+)$, with
    $\norm{u}_{\LL\infty(\R \times \Omega)} \le U$}. Then, for every
  $T>0$, there exists a unique solution
  $\left(\rho_1, \rho_2, v\right)$
  to~\eqref{eq:PDE-model-3x3}-\eqref{eq:initial-condition}-\eqref{eq:boundary-condition}
  on the time interval $[0, T]$, in the sense of~\Cref{def:sol}.

  Moreover, for a.e. $(t, x) \in [0, T] \times \Omega$, we have the
  following estimates:
  \begin{align}
    \label{eq:general-growth-est-rho1}
    0
    & \le \rho_1(t, x) \le \norm{\rho_{1,o}}_{\LL\infty}
      e^{\left|\alpha - \delta_1\right| T},
    \\
    \label{eq:general-growth-est-rho2}
    0
    & \le \rho_2(t, x) \le \left(\norm{\rho_{1,o}}_{\LL\infty}
      + \norm{\rho_{2,o}}_{\LL\infty}\right)
      e^{\alpha T},
    \\
    \label{eq:general-growth-est-v}
    0
    & \le v(t, x) \le \norm{v_{o}}_{\LL\infty}
      + \left( b \delta_2 \left(\norm{\rho_{1,o}}_{\LL\infty}
      + \norm{\rho_{2,o}}_{\LL\infty}\right)
      e^{\alpha T} + U \right) T.
  \end{align}
\end{theorem}

\begin{proof}
  Define
  \begin{equation*}
    \overline T = \sup \left\{T > 0: \textrm{ the solution
        to~\eqref{eq:PDE-model-3x3} exists in }[0, T]\right\}.
  \end{equation*}
  Clearly~\Cref{thm:local_wellposedness} implies that
  $\overline T > 0$.  Assume by contradiction that
  $\overline T < +\infty$.  Since $\rho_1$ and $v$ are positive
  by~\eqref{eq:positivity-solution} and since $\beta > 0$, then
  $\rho_1$ is a subsolution to
  \begin{equation*}
    \partial_t \rho_1 \le \Delta \rho_1 + \left(\alpha - \delta_1\right)
    \rho_1.
  \end{equation*}
  Hence~\Cref{prop:a-priori-growth} implies that
  \begin{equation}
    \label{eq:growth-est-rho1}
    0 \le \rho_1(t, x) \le \norm{\rho_{1,o}}_{\LL\infty}
    e^{\left|\alpha - \delta_1\right| \overline T},
  \end{equation}
  for every $t < \overline T$ and for a.e.~$x \in \Omega$.

  Consider now the equation for the sum $\rho_1 + \rho_2$:
  \begin{equation*}
    \pt (\rho_1 + \rho_2) = \Delta (\rho_1 + \rho_2)
    + (\alpha - \delta_1) \rho_1 - \delta_2 \, \rho_2.
  \end{equation*}
  Since $\alpha > 0$, $\delta_1 > 0$, $\delta_2 > 0$, and
  $\rho_2 \ge 0$ by~\eqref{eq:positivity-solution}, $\rho_1 + \rho_2$
  is a subsolution to
  \begin{equation*}
    \pt (\rho_1 + \rho_2)
    \leq
    \Delta (\rho_1 + \rho_2) + \alpha (\rho_1 + \rho_2).
  \end{equation*}
  Hence~\Cref{prop:a-priori-growth} implies that
  \begin{equation*}
    0 \le \rho_1(t, x) + \rho_2(t, x) \le \left(\norm{\rho_{1,o}}_{\LL\infty}
      + \norm{\rho_{2,o}}_{\LL\infty}\right)
    e^{\alpha \overline T},
  \end{equation*}
  for every $t < \overline T$ and for a.e.~$x \in \Omega$, so that
  \begin{equation}
    \label{eq:growth-est-rho2}
    0 \le \rho_2(t, x) \le \left(\norm{\rho_{1,o}}_{\LL\infty}
      + \norm{\rho_{2,o}}_{\LL\infty}\right)
    e^{\alpha \overline T},
  \end{equation}
  for every $t < \overline T$ and for a.e. $x \in \Omega$.

  Using the estimates~\eqref{eq:growth-est-rho1}
  and~\eqref{eq:growth-est-rho2}, we deduce that $v$ is subsolution to
  \begin{equation*}
    \pt v \le \Delta v + b \delta_2 \left(\norm{\rho_{1,o}}_{\LL\infty}
      + \norm{\rho_{2,o}}_{\LL\infty}\right)
    e^{\alpha \overline T} + U,
  \end{equation*}
  where we used the fact that $b > 0$, $\delta_2 > 0$, $B > 0$,
  $\delta_v > 0$, $\norm{u}_{\LL\infty} \le U$, and $\rho_1 \ge 0$,
  $v \ge 0$ by~\eqref{eq:positivity-solution}.
  Hence~\Cref{prop:a-priori-growth} implies that
  \begin{equation}
    \label{eq:growth-est-v}
    0 \le v(t, x) \le \norm{v_{o}}_{\LL\infty}
    + \left( b \delta_2 \left(\norm{\rho_{1,o}}_{\LL\infty}
        + \norm{\rho_{2,o}}_{\LL\infty}\right)
      e^{\alpha \overline T} + U \right) \overline T,
  \end{equation}
  for every $t < \overline T$ and for a.e.~$x \in \Omega$.

  Standard arguments in Sobolev spaces dependent on time together with
  the estimates~\eqref{eq:growth-est-rho1},
  \eqref{eq:growth-est-rho2}, and~\eqref{eq:growth-est-v} and the
  assumptions on $\Omega$ permit to extend $\rho_1$, $\rho_2$, and $v$
  by continuity at time $\overline T$.  Since
  $\rho_1\left(\overline T\right) \in \LL2\left(\Omega\right)$,
  $\rho_2\left(\overline T\right) \in \LL2\left(\Omega\right)$, and
  $v\left(\overline T\right) \in \LL2\left(\Omega\right)$,
  \Cref{thm:local_wellposedness} implies that the solution
  $\left(\rho_1, \rho_2, v\right)$ exists also for times bigger than
  $\overline T$. This is in contradiction with the definition of
  $\overline T$.  Finally the
  estimates~\eqref{eq:general-growth-est-rho1},
  \eqref{eq:general-growth-est-rho2},
  and~\eqref{eq:general-growth-est-v} easily follow
  from~\eqref{eq:growth-est-rho1}, \eqref{eq:growth-est-rho2},
  and~\eqref{eq:growth-est-v}.
\end{proof}



\section{Dependence from the control \texorpdfstring{$u$}{u}}
\label{sec:dep-sol-control}

In this part we show that the solution to~\eqref{eq:PDE-model-3x3}
continuously depends on the control $u$, viewed as a function in
{$\LL2\left([0, T]\times \Omega; [0, U]\right)$} endowed with
the strong topology.

\begin{theorem}
  \label{thm:dependence-from-u-H1}
  Assume $\alpha$, $\beta$, $\delta_1$, $\delta_2$, $\delta_v$, $b$,
  $B$, and $U$ fixed positive constants.  Let $\Omega \subseteq \R^N$
  be an open, connected, and bounded domain, with smooth boundary
  $\partial \Omega$. Fix $T > 0$,
  $\rho_{1,o}, \rho_{2,o}, v_{o} \in \LL\infty\left(\Omega; \R_+\right)$ and
  {$\bar u, \tilde u \in \LL\infty(\R \times \Omega; [0, U])$}.
  Define $\left(\bar \rho_1, \bar \rho_2, \bar v\right)$ and
  $\left(\tilde \rho_1, \tilde \rho_2, \tilde v\right)$ the solutions
  to~\eqref{eq:PDE-model-3x3}-\eqref{eq:initial-condition}-\eqref{eq:boundary-condition}
  on the time interval $[0, T]$ with controls $\bar u$ and $\tilde u$
  respectively.

  Then there exists a positive constant $C$, depending on the norms
  $\norm{\rho_{1,o}}_{\LL\infty\left(\Omega\right)}$,
  $\norm{\rho_{2,o}}_{\LL\infty\left(\Omega\right)}$,
  $\norm{v_{o}}_{\LL\infty\left(\Omega\right)}$, on $T$, and on the
  constants $\alpha$, $\delta_1$, $B$, $\beta$, $b$, $\delta_2$,
  $\delta_v$ such that, for every $t \in [0, T]$,
  \begin{equation}
    \label{eq:dependence-from-u}
    \begin{split}
      & \norm{\bar \rho_1(t) - \tilde
        \rho_1(t)}_{\LL2\left(\Omega\right)} + \norm{\bar \rho_2(t) -
        \tilde \rho_2(t)}_{\LL2\left(\Omega\right)} + \norm{\bar v(t)
        - \tilde v(t)}_{\LL2\left(\Omega\right)}
      \\
      \le & C \,{\norm{\bar u - \tilde u}_{\LL2\left((0,t)
            \times \Omega\right)}},
    \end{split}
  \end{equation}
  and
  \begin{equation}
    \label{eq:dependence-from-u-H1}
    \begin{split}
      & \int_0^t \!\!\left( \norm{\bar \rho_1(\tau) \!-\! \tilde
          \rho_1(\tau)}_{\HH1\left(\Omega\right)}^2 \!+\! \norm{\bar
          \rho_2(\tau) - \tilde
          \rho_2(\tau)}_{\HH1\left(\Omega\right)}^2 \!+\! \norm{\bar
          v(\tau) - \tilde v(\tau)}_{\HH1\left(\Omega\right)}^2
      \right) \!\dd\tau
      \\
      \le \ & C {\norm{\bar u - \tilde u}^2_{\LL2\left((0,t)
            \times \Omega\right)}}.
    \end{split}
  \end{equation}
\end{theorem}
\begin{proof}
  Fix two control functions
  {$\bar u, \tilde u \in \LL\infty\left([0, T] \times \Omega;
      [0, U]\right)$} and denote by
  $\left(\bar \rho_1, \bar \rho_2, \bar v\right)$ and by
  $\left(\tilde \rho_1, \tilde \rho_2, \tilde v\right)$ the
  corresponding solutions
  to~\eqref{eq:PDE-model-3x3}-\eqref{eq:initial-condition}-\eqref{eq:boundary-condition}.

  Consider the difference between the equations for $\bar v$ and
  $\tilde v$ in~\eqref{eq:PDE-model-3x3}, and rearrange it as follows
  \begin{displaymath}
    \pt \!\left(\bar v - \tilde v\right) =
    \Delta \left(\bar v - \tilde v\right)
    - \left(\delta_v + B \bar \rho_1\right) \left(\bar v - \tilde v\right)
    + b \delta_2 \left(\bar \rho_2 \!-\! \tilde \rho_2\right)
    -B \tilde v \left(\bar \rho_1 \!-\! \tilde \rho_1\right)
    + \bar u - \tilde u.
  \end{displaymath}
  Observe first that the bilinear form appearing above is weakly
  coercive, so in particular
  \begin{equation}
    \label{eq:bilin-coercive-v}
    \begin{split}
      \int_\Omega \left[ \nabla \left(\bar v - \tilde v\right) \cdot
        \nabla \left(\bar v - \tilde v\right) + \left(\delta_v + B
          \bar \rho_1\right) \left(\bar v - \tilde v\right)^2 \right]
      \dd{x} &
      \\
      + \left(\frac12 + \delta_v + B \norm{\bar \rho_1
          (t)}_{\LL\infty\left(\Omega\right)}\right) \norm{\bar v -
        \tilde v}_{\LL2\left(\Omega\right)}^2 & \ge \
      \frac12\norm{\bar v - \tilde v}^2_{\HH1 \left(\Omega\right)}.
    \end{split}
  \end{equation}
  By~\Cref{def:sol}, Point 5., we have that, for a.e. $t \in [0, T]$
  and for all $w \in \HH1 (\Omega)$,
  \begin{align*}
    & \int_\Omega \left(\dot{\bar v} - \dot{\tilde v}\right) w \dd x
      + \int_\Omega \nabla \left(\bar v - \tilde v\right) \cdot
      \nabla w \dd x
      +\int_\Omega \left(\delta_v + B \bar \rho_1\right)
      \left(\bar v - \tilde v \right) w \dd{x}
    \\
    = \
    & b \delta_2 \int_\Omega \left(\bar \rho_2 - \tilde \rho_2\right) w
      \dd x
      - B \int_\Omega \tilde v \left(\bar \rho_1 - \tilde \rho_1\right) w \dd x
      +{\int_\Omega \left(\bar u -\tilde u \right)  w \dd x} .
  \end{align*}
  Taking $w(x) = \left(\bar v(t, x) \!-\! \tilde v(t, x)\right)$ in
  the previous equation and exploiting~\eqref{eq:bilin-coercive-v}, we
  deduce that, for a.e. $t \in [0, T]$,
  \begin{align*}
    & \frac12 \, \frac{\dd{}}{\dd t}
      \norm{\bar v(t) - \tilde v(t)}_{\LL2\left(\Omega\right)}^2
      + \frac12\, \norm{\bar v(t) - \tilde v(t)}
      _{\HH1\left(\Omega\right)}^2
    \\
    \le \
    &  b \delta_2 \int_\Omega \left(\bar \rho_2(t, x)
      - \tilde \rho_2(t, x)\right) \left(\bar v(t, x) - \tilde v(t, x)\right)
      \dd x
    \\
    & - B \int_\Omega \tilde v(t, x) \left(\bar \rho_1(t, x)
      - \tilde \rho_1(t, x)\right) \left(\bar v(t, x)
      - \tilde v(t, x)\right) \dd x
    \\
    & + {\int_\Omega \left(\bar u(t,x) - \tilde u(t,x)\right)
      \left(\bar v(t, x) - \tilde v(t,x)\right) \dd x}
    \\
    & + \left(\frac12 + \delta_v + B
      \norm{\bar \rho_1 (t)}_{\LL\infty\left(\Omega\right)}\right)
      \norm{\bar v (t) -\tilde v (t)}^2_{\LL2\left(\Omega\right)}
    \\
    \le \
    &  \frac{b \, \delta_2}{2}  \left(
      \norm{\bar \rho_2 (t) - \tilde \rho_2(t)}^2_{\LL2 (\Omega)}
      +
      \norm{\bar v (t) - \tilde v (t)}^2_{\LL2 (\Omega)}
      \right)
    \\
    & + \frac{B}{2} \left(
      \norm{\tilde v(t)}_{\LL\infty (\Omega)}^2
      \norm{\bar v (t) - \tilde v(t)}^2_{\LL2 (\Omega)}
      +
      \norm{\bar \rho_1 (t) - \tilde \rho_1(t)}^2_{\LL2 (\Omega)}
      \right)
    \\
    & +{\frac12
      \norm{\bar u(t) - \tilde u (t)}^2_{\LL2 (\Omega)}
      + \frac1{2}
      \norm{\bar v (t) - \tilde v(t)}^2_{\LL2 (\Omega)}}
    \\
    & + \left(\frac12 + \delta_v + B
      \norm{\bar \rho_1 (t)}_{\LL\infty\left(\Omega\right)}\right)
      \norm{\bar v (t) -\tilde v (t)}^2_{\LL2\left(\Omega\right)}.
  \end{align*}

  We proceed similarly for $\bar \rho_1$ and $\tilde \rho_1$:
  \begin{displaymath}
    \pt \! \left(\bar \rho_1 - \tilde \rho_1\right)
    =
    \Delta \left(\bar \rho_1 - \tilde \rho_1\right)
    + \left(\alpha - \delta_1\right)\left(\bar \rho_1 - \tilde \rho_1\right)
    - \beta \bar v \left(\bar \rho_1 - \tilde \rho_1\right)
    - \beta \tilde \rho_1 \left(\bar v - \tilde v\right).
  \end{displaymath}
  For a.e. $t \in [0, T]$ and for all $w \in \LL2\left(\Omega\right)$,
  we have that
  \begin{align}
    \nonumber
    & \int_\Omega \left(\dot{\bar \rho}_1 - \dot{\tilde \rho}_1\right) w \dd{x}
      + \int_\Omega \nabla \left( \bar\rho_1  - \tilde \rho_1 \right)
      \cdot \nabla w \dd{x}
    \\ \label{eq:5}
    & - (\alpha - \delta_1)
      \int_\Omega\left( \bar\rho_1  - \tilde \rho_1 \right)  w \dd{x}
      + \beta \int_\Omega \bar v \left(\bar \rho_1 - \tilde \rho_1\right) w\dd{x}
    \\ \nonumber
    = \
    & - \beta \int_\Omega \tilde \rho_1 \left(\bar v - \tilde v \right) w \dd{x}.
  \end{align}
  Observe that the bilinear form is weakly coercive:
  \begin{equation}
    \label{eq:bilin-coercive-rho1}
    \begin{split}
      \int_\Omega \nabla \left( \bar\rho_1 - \tilde \rho_1 \right)
      \cdot \nabla \left( \bar\rho_1 - \tilde \rho_1 \right) \dd{x} +
      \int_\Omega\left(\delta_1 - \alpha + \beta \bar v \right)\left(
        \bar\rho_1 - \tilde \rho_1 \right)^2\dd{x}
      \\
      + \left(\frac12 + \delta_1 + \alpha + \beta \norm{\bar
          v}_{\LL\infty\left(\Omega\right)}\right) \norm{\bar \rho_1 -
        \tilde \rho_1}^2_{\LL2\left(\Omega\right)} \ge
      \frac12\norm{\bar \rho_1 - \tilde
        \rho_1}^2_{\HH1\left(\Omega\right)}.
    \end{split}
  \end{equation}
  Taking $w(x) = \left(\bar\rho_1(t,x) - \tilde \rho_1 (t,x) \right)$
  in~\eqref{eq:5} and exploiting~\eqref{eq:bilin-coercive-rho1} yields
  \begin{align*}
    & \frac12 \, \frac{\dd{}}{\dd t}
      \norm{\bar \rho_1(t) - \tilde \rho_1(t)}_{\LL2\left(\Omega\right)}^2
      +  \frac12 \norm{ \bar\rho_1 (t)  - \tilde \rho_1 (t) }_{\HH1 (\Omega)}^2
    \\
    \le \
    & -  \beta
      \int_\Omega \tilde \rho_1
      \left(\bar \rho_1 (t,x) - \tilde \rho_1 (t,x)\right)
      \left(\bar v (t,x) -  \tilde v (t,x)\right) \dd{x}
    \\
    & +  \left(\frac12 + \delta_1 + \beta \norm{\bar
      v (t)}_{\LL\infty\left(\Omega\right)}\right) \norm{\bar \rho_1 (t) -
      \tilde \rho_1 (t)}^2_{\LL2\left(\Omega\right)}
    \\
    \le \
    & \frac{\beta}{2} \norm{\tilde \rho_1(t)}_{\LL\infty (\Omega)}^2
      \norm{\bar \rho_1(t) - \tilde \rho_1(t)}_{\LL2\left(\Omega\right)}^2
      + \frac{\beta}{2}  \norm{\bar v(t) - \tilde v(t)}_{\LL2\left(\Omega\right)}^2
    \\
    & + \left(\frac12 + \delta_1 + \alpha+ \beta \norm{\bar
      v (t)}_{\LL\infty\left(\Omega\right)}\right) \norm{\bar \rho_1 (t) -
      \tilde \rho_1 (t)}^2_{\LL2\left(\Omega\right)}.
  \end{align*}

  Finally, the same arguments can be applied to $\bar \rho_2$ and
  $\tilde \rho_2$:
  \begin{displaymath}
    \pt \! \left(\bar \rho_2 - \tilde \rho_2\right)
    =
    \Delta \left(\bar \rho_2 - \tilde \rho_2\right)
    -\delta_2 \left(\bar \rho_2 - \tilde \rho_2\right)
    +\beta \bar \rho_1 \bar v - \beta \tilde \rho_1 \tilde v.
  \end{displaymath}
  For a.e. $t \in [0, T]$ and for every
  $w \in \LL2\left(\Omega\right)$ it holds
  \begin{equation}
    \label{eq:3}
    \begin{split}
      & \int_\Omega \left( \dot{\bar \rho}_2 - \dot{\tilde
          \rho}_2\right) w \dd{x} +\int_\Omega \nabla \left( \bar
        \rho_2 - \tilde \rho_2\right) \cdot \nabla w \dd{x} + \delta_2
      \int_\Omega \left( \bar \rho_2 - \tilde \rho_2\right) w \dd{x}
      \\
      = \ & \beta \int_\Omega \left( \bar \rho_1 \, \bar v - \tilde
        \rho_1 \, \tilde v\right) w \dd{x}.
    \end{split}
  \end{equation}
  The bilinear form appearing above is weakly coercive:
  \begin{equation}
    \label{eq:bilin-coercive-rho2}
    \begin{split}
      \int_\Omega \nabla \left( \bar \rho_2 - \tilde \rho_2\right)
      \cdot \nabla \left( \bar \rho_2 - \tilde \rho_2\right) \dd{x} +
      \delta_2 \int_\Omega \left( \bar \rho_2 - \tilde \rho_2\right)^2
      \dd{x}&
      \\
      + \left(\frac12 + \delta_2\right) \norm{\bar \rho_2 - \tilde
        \rho_2}^2_{\LL2 \left(\Omega\right)} & \ge \ \frac12
      \norm{\bar \rho_2 - \tilde \rho_2}^2_{\HH1 \left(\Omega\right)}.
    \end{split}
  \end{equation}
  Taking $w(x) = \left(\bar \rho_2 (t,x) - \tilde \rho_2 (t,x)\right)$
  in~\eqref{eq:3} and exploiting~\eqref{eq:bilin-coercive-rho2}, we
  get
  \begin{align*}
    & \frac12 \frac{\dd{}}{\dd t}
      \norm{\bar \rho_2(t) - \tilde \rho_2(t)}_{\LL2\left(\Omega\right)}^2
      + \frac12 \norm{\bar\rho_2 (t)  - \tilde \rho_2 (t)}_{\HH1 (\Omega)}^2
    \\
    \le \
    &  \beta \int_\Omega \bar \rho_1 (t,x)
      \left( \bar v (t,x) - \tilde v(t,x)\right)
      \left(\bar \rho_2(t,x) - \tilde \rho_2(t,x)\right) \dd{x}
    \\
    & + \beta \int_\Omega \tilde v (t,x)
      \left(\bar \rho_1(t,x) - \tilde \rho_1(t,x)\right)
      \left(\bar \rho_2(t,x) - \tilde \rho_2(t,x)\right) \dd{x}
    \\
    & + \left(\frac12 + \delta_2\right)
      \norm{\bar \rho_2 (t) - \tilde \rho_2 (t)}^2_{\LL2 \left(\Omega\right)}
    \\
    \le \
    & \frac{\beta}{2} \left( \norm{\bar \rho_1(t)}_{\LL\infty (\Omega)}^2
      \norm{\bar \rho_2(t) - \tilde \rho_2(t)}_{\LL2\left(\Omega\right)}^2
      +
      \norm{\bar v(t) - \tilde v(t)}_{\LL2\left(\Omega\right)}^2
      \right.
    \\
    & \left.+\norm{\tilde v(t)}_{\LL\infty (\Omega)}^2
      \norm{\bar \rho_1(t) - \tilde \rho_1(t)}_{\LL2\left(\Omega\right)}^2
      +
      \norm{\bar \rho_2(t) - \tilde \rho_2(t)}_{\LL2\left(\Omega\right)}^2
      \right)
    \\
    & + \left(\frac12 + \delta_2\right)
      \norm{\bar \rho_2 (t) - \tilde \rho_2 (t)}^2_{\LL2 \left(\Omega\right)}.
  \end{align*}

  Thus, collecting together the estimates obtained above, for a.e.
  $t \in [0, T]$, we get
  \begin{align*}
    &  \frac{\dd{}}{\dd t}
      \left(
      \norm{\bar \rho_1(t) - \tilde \rho_1(t)}_{\LL2\left(\Omega\right)}^2
      +
      \norm{\bar \rho_2(t) - \tilde \rho_2(t)}_{\LL2\left(\Omega\right)}^2
      +
      \norm{\bar v(t) - \tilde v(t)}_{\LL2\left(\Omega\right)}^2
      \right)
    \\
    & \quad
      +
      \norm{\bar \rho_1(t) - \tilde \rho_1(t)}_{\HH1\left(\Omega\right)}^2
      +
      \norm{\bar \rho_2(t) - \tilde \rho_2(t)}_{\HH1\left(\Omega\right)}^2
      +
      \norm{\bar v(t) - \tilde v(t)}_{\HH1\left(\Omega\right)}^2
    \\
    \le \
    &
      \left(B
      + \beta \norm{\tilde \rho_1(t)}_{\LL\infty (\Omega)}^2
      + \beta \norm{\tilde v(t)}_{\LL\infty (\Omega)}^2
      +1+2 \, \delta_1 + 2 \, \alpha + 2 \, \beta \norm{\bar v (t)}_{\LL\infty\left(\Omega\right)}
      \right)
      \\
      & \times
      \norm{\bar \rho_1(t) - \tilde \rho_1(t)}_{\LL2\left(\Omega\right)}^2
    \\
    &+
      \left( \beta + \beta \norm{\bar \rho_1(t)}_{\LL\infty (\Omega)}^2
      + b \, \delta_2 +1 + 2\,\delta_2 \right)
      \norm{\bar \rho_2(t) - \tilde \rho_2(t)}_{\LL2\left(\Omega\right)}^2
    \\
    & + \left( 2 \, \beta + b \, \delta_2
      + B \norm{\tilde v(t)}_{\LL\infty (\Omega)}^2
      +{2}
      +2\, \delta_v + B \norm{\bar \rho_1 (t)}_{\LL\infty\left(\Omega\right)}\right)
      \\
      & \times
      \norm{\bar v(t) - \tilde v(t)}_{\LL2\left(\Omega\right)}^2
    \\
    & +
      {\norm{\bar u(t) - \tilde u (t)}^2_{\LL2 (\Omega)}}.
  \end{align*}
  Set
  \begin{align*}
    C(t) = \
    \max\Big\{
    &
      B
      + \beta \left(\norm{\tilde \rho_1(t)}_{\LL\infty (\Omega)}^2
      + \norm{\tilde v(t)}_{\LL\infty (\Omega)}^2
      + 2 \, \norm{\bar v (t)}_{\LL\infty\left(\Omega\right)}\right)
    \\
    &
      +1+2 \, \delta_1 + 2\,
      \\
      & \alpha, \,\beta + \beta \norm{\bar \rho_1(t)}_{\LL\infty (\Omega)}^2
      + b \, \delta_2 +1 + 2\,\delta_2, \,
    \\
    & 2 \, \beta + b \, \delta_2
      + B \norm{\tilde v(t)}_{\LL\infty (\Omega)}^2
      +{2}+2\, \delta_v + B \norm{\bar \rho_1 (t)}_{\LL\infty\left(\Omega\right)}
      \Big\}.
  \end{align*}
  An application of Gronwall's inequality yields the following bound
  for the $\LL2$-norm:
  \begin{align*}
    &  \norm{\bar \rho_1(t) - \tilde \rho_1(t)}_{\LL2\left(\Omega\right)}^2
      +
      \norm{\bar \rho_2(t) - \tilde \rho_2(t)}_{\LL2\left(\Omega\right)}^2
      +
      \norm{\bar v(t) - \tilde v(t)}_{\LL2\left(\Omega\right)}^2
    \\
    \le \
    &
      \int_0^t {\norm{\bar u(s) - \tilde u (s)}^2_{\LL2\left(\Omega\right)}}
      \exp\left({\int_s^t C(\tau) \dd{\tau}}\right) \dd s.
  \end{align*}
  Moreover, we get
  \begin{align*}
    & \int_0^t \left(
      \norm{\bar \rho_1(\tau) - \tilde \rho_1(\tau)}_{\HH1\left(\Omega\right)}^2
      +
      \norm{\bar \rho_2(\tau) - \tilde \rho_2(\tau)}_{\HH1\left(\Omega\right)}^2
      +
      \norm{\bar v(\tau) - \tilde v(\tau)}_{\HH1\left(\Omega\right)}^2
      \right)
      \dd\tau
    \\
    \le \
    & \left(1+ \int_0^t C (\tau)
      \exp\left({\int_\tau^s C(s) \dd{s}}\right) \dd\tau\right)
      \int_0^t {\norm{\bar u(\tau) - \tilde u (\tau)}^2_{\LL2\left(\Omega\right)}} \dd\tau
    \\
    = \
    & \exp\left(\int_0^t C (\tau) \dd\tau\right)
      \int_0^t {\norm{\bar u(\tau) - \tilde u (\tau)}^2_{\LL2\left(\Omega\right)}} \dd\tau.
  \end{align*}
  Note that, using the estimates~\eqref{eq:general-growth-est-rho1},
  \eqref{eq:general-growth-est-rho2},
  and~\eqref{eq:general-growth-est-v} of~\Cref{thm:global-existence},
  the function $C(t)$ can be controlled by a constant depending on
  $T$, on the initial conditions, on $\abs{\Omega}$, and on the
  various constant appearing on system~\eqref{eq:PDE-model-3x3}.
  Therefore the estimates~\eqref{eq:dependence-from-u}
  and~\eqref{eq:dependence-from-u-H1} hold and the proof is finished.
\end{proof}

\section{Optimal control problem}
\label{sec:optimal-control}

In this section we consider optimal control problems for the
system~\eqref{eq:PDE-model-3x3}-\eqref{eq:initial-condition}-\eqref{eq:boundary-condition},
obtained through the minimization of a functional, which explicitly
depends on the control $u$ and consequently on the solution
to~\eqref{eq:PDE-model-3x3}.  To this aim, define the (continuous)
functions
\begin{equation*}
  \psi_1: \R^3 \to \R,
  \qquad
  \psi_2: \R^4 \to \R,
\end{equation*}
and the functional
{$J: \LL2\left([0, T] \times \Omega; [0, U]\right) \to \R$} as
\begin{equation}
  \label{eq:funct-J}
  \begin{split}
    J(u) & = \int_\Omega \psi_1\left(\rho_1(T, x), \rho_2(T, x), v(T,
      x)\right) \dd x
    \\
    & \quad + \int_0^T \int_\Omega \psi_2\left(\rho_1(t, x), \rho_2(t,
      x), v(t, x), {u(t,x)}\right) \dd x \dd t,
  \end{split}
\end{equation}
which we aim to minimize.




\begin{remark}
\label{rmk:functionals}
  The general form of the functional~\eqref{eq:funct-J} describes in a
  unified way several possible objective functionals, which are
  natural in application related to cancer therapies against glioma.

  For example, if the main objective is the maximal reduction of the
  volume of the cancer at a time $T$, one can consider the functional
  \begin{equation*}
    J(u) = \gamma_1 \int_\Omega \rho_1(T, x) \dd x
    + \gamma_2 \int_\Omega \rho_2(T, x) \dd x,
  \end{equation*}
  for suitable $\gamma_1, \gamma_2 \ge 0$. This is a special version
  of~\eqref{eq:funct-J}, obtained with the choice
  $\psi_1(\rho_1, \rho_2, v) = \gamma_1 \rho_1 + \gamma_2 \rho_2$ and
  $\psi_2(\rho_1, \rho_2, v, u) = 0$.

  Another similar example is derived when the objective is the minimal
  cancer size at a time $T$, obtained with the least dose of
  treatment, due to its side effects. In this case a meaningful
  functional is
  \begin{equation*}
    J(u) = \gamma_1 \int_\Omega \rho_1(T, x) \dd x
    + \gamma_2 \int_\Omega \rho_2(T, x) \dd x
    + {\int_0^T \int_\Omega u^p(t,x) \dd x \dd t},
  \end{equation*}
  where $\gamma_1, \gamma_2 \ge 0$ and $p \ge 1$.  This is a special
  version of~\eqref{eq:funct-J}, obtained with the choice
  $\psi_1(\rho_1, \rho_2, v) = \gamma_1 \rho_1 + \gamma_2 \rho_2$ and
  {$\psi_2(\rho_1, \rho_2, v, u) = u^p$.}

  A third example is related to a possible way to make the glioma a
  chronic disease. In this case, if
  $\bar \rho \in \LL2\left(\Omega\right)$ denotes the distribution of
  the cancer in a chronic situation, then one aims to minimize
  \begin{equation*}
    \begin{split}
      J(u) & = \int_\Omega \left(\rho_1(T, x) - \bar \rho(x)\right)^2
      \dd x + \int_0^T \int_\Omega \left(\rho_1(t, x) - \bar
        \rho(x)\right)^2 \dd x \dd t
      \\
      & \quad + {\int_0^T\int_\Omega u^p(t,x)\dd x \dd t}.
    \end{split}
  \end{equation*}
  This is~\eqref{eq:funct-J} with
  $\psi_1(\rho_1, \rho_2, v) = (\rho_1 - \bar \rho)^2$ and
  $\psi_2(\rho_1, \rho_2, v, u) = (\rho_1 - \bar \rho)^2 +
  {u^p}$.
\end{remark}

Existence of optimal controls is guaranteed by the next result, whose
proof is based on the direct method of the calculus of variation; see
for example~\cite{MR2361288, MR1962933} and the references therein.
\begin{theorem}
  \label{thm:existence-minimum-2}
  Let $\alpha$, $\beta$, $\delta_1$, $\delta_2$, $\delta_v$, $b$,
  $B$, and $U$ be fixed positive constants.  Assume that
  $\psi_1$ is continuous, convex and $\psi_1(z) \ge 0$ for every $z
  \in \R^3$. Suppose moreover that
  \begin{displaymath}
    \psi_2(\rho_1,\rho_2,v,u) =
    \psi_{2,1} (\rho_1,\rho_2,v) + \psi_{2,2}(\rho_1,\rho_2,v) u^p,
  \end{displaymath}
  where $p \ge 1$ and $\psi_{2,1}, \psi_{2,2}: \R^3 \to
  \R$ are continuous and positive functions.  Let $\mathcal U \ne
  \emptyset$ be a closed (with respect to the strong topology) and
  convex subset of {$\LL2\left((0, T) \times \Omega; [0,
      U]\right)$}.

  Then there exists $\bar u \in \mathcal U$ such that
  \begin{equation}
    \label{eq:optimal-condition}
    J\left(\bar u\right) = \min_{u \in \mathcal U} J(u).
  \end{equation}
\end{theorem}

\begin{proof}
  Consider a minimizing sequence $u_n$ for the functional $J$, i.e.  a
  sequence $u_n \in \mathcal U$ such that
  \begin{equation*}
    \lim_{n \to +\infty} J\left(u_n\right) = \inf_{u \in \mathcal U} J(u).
  \end{equation*}
  This is possible, since $J(u) \ge 0$ for every $u \in \mathcal U$ by
  assumptions on $\psi_1$, $\psi_{2,1}$, $\psi_{2,2}$, and since
  $J(0) < +\infty$, which is a consequence of the estimates
  in~\Cref{thm:global-existence}.  We clearly have that
  {\begin{equation*} \norm{u_n}_{\LL2\left((0, T) \times \Omega
        \right)}^2 = \int_0^T \int_\Omega\abs{u_n(t,x)}^2 \dd x \dd t
      \le U^2 \modulo{\Omega} T,
    \end{equation*}
    where $\modulo{\Omega}$ denotes the Lebesgue measure of $\Omega$,}
  so that there exists $\bar u \in \mathcal U$ (here $\mathcal U$
  is also closed in the
  weak topology, since it is convex) and a subsequence $u_{n_h}$ such
  that {$u_{n_h} \wc \bar u$ in
    $\LL2\left((0, T)\times \Omega; [0, U]\right)$}.  Without loss of
  generality, we assume that the whole sequence $u_n$ weakly converges
  in {$\LL2\left((0, T) \times \Omega; [0, U]\right)$} to
  $\bar u$.  Note, moreover, that there is weak convergence to the
  same $\bar u$ in every
  {$\LL{p}((0,T) \times \Omega; [0, U])$}, $p \ge 1$.

  For every $n \in \N$, denote with
  $\left(\rho_1^n, \rho_2^n, v^n\right)$ the solution
  to~\eqref{eq:PDE-model-3x3} corresponding to the control $u_n$, in
  the sense of~\Cref{def:sol}, which exists by~\Cref{thm:global-existence}.  Moreover define
  $\left(\bar \rho_1, \bar \rho_2, \bar v\right)$ the solution
  to~\eqref{eq:PDE-model-3x3} corresponding to the control $\bar u$.

  Define, for every $n \in \N$, $S_n = \rho_1^n v^n$.
  By~\eqref{eq:general-growth-est-rho1}
  and~\eqref{eq:general-growth-est-v}, we deduce that
  \begin{align*}
    \abs{S_n(t, x)}
    & \le \norm{\rho_{1,o}}_{\LL\infty}
      \norm{v_{o}}_{\LL\infty}
      e^{\left|\alpha - \delta_1\right| T}
    \\
    & \quad
      + \norm{\rho_{1,o}}_{\LL\infty}
      \left( b \delta_2 \left(\norm{\rho_{1,o}}_{\LL\infty}
      + \norm{\rho_{2,o}}_{\LL\infty}\right)
      e^{\alpha T} + U \right) T
  \end{align*}
  for a.e. $(t, x) \in [0, T] \times \Omega$. Therefore, without loss
  of generality, there exists
  $\bar S \in \LL2\left((0, T) \times \Omega\right)$ so that $S_n$
  weakly converges to $\bar S$ in
  $\LL2\left((0, T) \times \Omega\right)$.

  Note that, for every $n \in \N$, the triple
  $\left(\rho_1^n, \rho_2^n, v^n\right)$ is a solution to the linear
  system
  \begin{equation}
    \label{eq:PDE-model-linear}
    \left\{
      \begin{array}{l}
        \pt \rho_1 = \Delta \rho_1 + \left(\alpha - \delta_1\right) \rho_1
        - \beta S_n
        \\
        \pt \rho_2 = \Delta \rho_2 + \beta S_n - \delta_2 \rho_2
        \\
        \pt v = \Delta v + b \delta_2 \rho_2 - B S_n - \delta_v v + u_n
      \end{array}
    \right.
  \end{equation}
    in the sense of~\Cref{def:sol}.
    We apply~\Cref{cor:solution-operator} to each linear equation
    of~\eqref{eq:PDE-model-linear}. Since the operator defined
    in~\Cref{cor:solution-operator} is also continuous with respect to
    the weak topology of both domain and codomain,
    see~\cite[Theorem~3.10]{MR2759829}, then there exist
    $\tilde \rho_1$, $\tilde \rho_2$, and $\tilde v$ in
    $\HH1\left(0, T; \HH1\left(\Omega\right),
      \HH1\left(\Omega\right)^*\right)$ such that
  \begin{equation*}
    \rho_1^n \wc \tilde \rho_1,
    \qquad
    \rho_2^n \wc \tilde \rho_2,
    \quad \textrm{and}\quad
    v^n \wc \tilde v
  \end{equation*}
  weakly in
  $\HH1\left(0, T; \HH1\left(\Omega\right),
    \HH1\left(\Omega\right)^*\right)$ and the triple
  $\left(\tilde \rho_1, \tilde \rho_2, \tilde v\right)$ satisfies the
  linear system
  \begin{equation}
    \label{eq:PDE-model-linear-2}
    \left\{
      \begin{array}{l}
        \pt \tilde \rho_1 =
        \Delta \tilde \rho_1 + \left(\alpha - \delta_1\right) \tilde \rho_1
        - \beta \bar S
        \\
        \pt \tilde \rho_2 = \Delta \tilde \rho_2 + \beta \bar S
        - \delta_2 \tilde \rho_2
        \\
        \pt \tilde v = \Delta \tilde v + b \delta_2 \tilde \rho_2
        - B \bar S - \delta_v \tilde v + \bar u.
      \end{array}
    \right.
  \end{equation}
  Moreover, since
  $\HH1\left(0, T; \HH1\left(\Omega\right),
    \HH1\left(\Omega\right)^*\right)$ is continuously embedded in the space
  $\CC0\left([0, T]; \LL2\left(\Omega\right)\right)$
  (see~\cite[Theorem 7.104]{Salsa}), we also deduce that
  \begin{equation}
    \label{eq:weak-conv-trace}
    \rho_1^n(T) \wc \tilde \rho_1(T),
    \qquad
    \rho_2^n(T) \wc \tilde \rho_2(T),
    \quad \textrm{and}\quad
    v^n(T) \wc \tilde v(T)
  \end{equation}
  weakly in $\LL2\left(\Omega\right)$.  This is a consequence that the
  dual of $\CC0\left([0, T]; \LL2\left(\Omega\right)\right)$ can be
  identified with integral operators from $\CC0\left([0, T];\R\right)$
  to the dual space of $\LL2\left(\Omega\right)$ (see~\cite[sections
  3.2 and 3.5]{MR1888309}), which can be described through vector
  measures on $[0, T]$ over the dual of $\LL2\left(\Omega\right)$;
  see~\cite[Proposition 5.28]{MR1888309}.

  By~\cite[Theorem~10.1]{MR0241822}, for every $n \in \N$, the
  solution $\left(\rho_1^n, \rho_2^n, v^n\right)$ is Hölder continuous
  in each subset compactly embedded in
  $\left(0, T\right) \times \Omega$, with exponent not depending on
  $n$.  Therefore, by Ascoli-Arzelà Theorem, there exist continuous
  functions $\hat \rho_1$, $\hat \rho_2$, and $\hat v$ such that,
  possibly extracting a subsequence,
  \begin{equation*}
    \rho_1^n \to \hat \rho_1,
    \qquad
    \rho_2^n \to \hat \rho_2,
    \qquad
    v^n \to \hat v,
  \end{equation*}
  as $n\to +\infty$, and the convergence is uniform. Hence
  $\tilde \rho_1 = \hat \rho_1$, $\tilde \rho_2 = \hat \rho_2$,
  $\tilde v = \hat v$ and so
  \begin{equation*}
    S_n = \rho_1^n v^n \longrightarrow \hat \rho_1 \hat v
  \end{equation*}
  uniformly and in $\LL2\left(\left(0, T\right) \times \Omega\right)$.

  Therefore, the triple
  $\left(\hat \rho_1, \hat \rho_2, \hat v \right)$
  solves~\eqref{eq:PDE-model-linear-2} with
  $\bar S = \hat \rho_1 \hat v$ and control $\bar u$. Since
  by~\Cref{thm:global-existence} the solution to such problem is
  unique, it holds
  $\left(\hat \rho_1, \hat \rho_2, \hat v\right) = \left(\bar \rho_1,
    \bar \rho_2, \bar v\right)$.

  We now show that the control $\bar u$ is indeed optimal. Consider
  the three terms defining the functional $J$ separately.  The
  functional
  \begin{equation*}
    \begin{array}{ccc}
      \LL2\left(\Omega\right) \times
      \LL2\left(\Omega\right) \times \LL2\left(\Omega\right)
      &
        \longrightarrow
      &
        \R
      \\
      \left(\rho_1, \rho_2, v\right)
      & \longmapsto
      & \displaystyle
        \int_\Omega \psi_1\left(\rho_1(x), \rho_2(x), v(x)\right) \dd x
    \end{array}
  \end{equation*}
  is sequential lower semicontinuous with respect to the strong
  topology.  Since $\psi_1$ is a convex function, then it is also
  sequential lower semicontinuous with respect to the weak topology.
  Therefore, since~\eqref{eq:weak-conv-trace}, we deduce that
  \begin{equation*}
    \begin{split}
      & \liminf_{n \to +\infty} \int_\Omega \psi_1\left(\rho_1^n(T,
        x), \rho_2^n(T, x), v^n(T, x)\right) \dd x
      \\
      \ge & \int_\Omega \psi_1\left(\bar \rho_1(T, x), \bar \rho_2(T,
        x), \bar v(T, x)\right) \dd x.
    \end{split}
  \end{equation*}

  Moreover, since $\left(\rho_1^n, \rho_2^n, v^n\right)$ converges to
  $\left(\bar \rho_1, \bar \rho_2, \bar v\right)$ for a.e.
  $(t, x) \in (0, T) \times \Omega$, the
  estimates~\eqref{eq:general-growth-est-rho1}--\eqref{eq:general-growth-est-v}
  hold, and $\psi_{2,1}$ and $\psi_{2,2}$ are continuous functions,
  then the Dominated Convergence Theorem implies that
  \begin{equation*}
    \begin{split}
      & \lim_{n \to +\infty} \int_0^T \int_\Omega
      \psi_{2,1}\left(\rho_1^n(t, x), \rho_2^n(t, x), v^n(t, x)\right)
      \dd x\, \dd t
      \\
      = & \int_0^T \int_\Omega \psi_{2,1}\left(\bar \rho_1(t, x), \bar
        \rho_2(t, x), \bar v(t, x)\right) \dd x\, \dd t
    \end{split}
  \end{equation*}
  and also that
  \begin{equation}
    \label{eq:convergence-psi-22-LP}
    \lim_{n \to + \infty}\psi_{2,2}\left(\rho_1^n, \rho_2^n, v^n\right)
    = \psi_{2,2}\left(\bar \rho_1, \bar \rho_2, \bar v\right)
  \end{equation}
  in $\LL{p}\left(\Omega\right)$ for every $p \in [1, +\infty)$.

  Finally, by~\eqref{eq:convergence-psi-22-LP} and the fact that
  {$\abs{u_n^p(t,x)} \le U$ for all
    $(t,x) \in [0, T] \times \Omega$}, we deduce that
  \begin{align*}
    & \quad
      \liminf_{n \to +\infty} \int_0^T \int_{\Omega}
      \psi_{2,2}\left(\rho_1^n(t, x), \rho_2^n(t, x), v^n(t, x)\right)
      {u^p_n(t,x)}
      \dd x\, \dd t
    \\
    & = \liminf_{n \to +\infty} \int_0^T \int_{\Omega}
      \psi_{2,2}\left(\bar \rho_1(t, x),
      \bar \rho_2(t, x),
      \bar v(t, x)\right) {u^p_n(t,x)} \dd x\, \dd t
    \\
    & \quad + \lim_{n \to +\infty} \int_0^T \int_{\Omega}
      \big[\psi_{2,2}\left(\rho_1^n(t, x), \rho_2^n(t, x), v^n(t, x)\right)
    \\
    & \qquad\qquad\qquad\qquad
      - \psi_{2,2}\left(\bar \rho_1(t, x),
      \bar \rho_2(t, x),
      \bar v(t, x)\right) \big]  {u^p_n(t,x)} \dd x\, \dd t
    \\
    & = \liminf_{n \to +\infty} \int_0^T \int_{\Omega}
      \psi_{2,2}\left(\bar \rho_1(t, x),
      \bar \rho_2(t, x),
      \bar v(t, x)\right)  {u^p_n(t,x)} \dd x\, \dd t.
  \end{align*}
  Since the functional
  \begin{equation*}
    \begin{array}{rcc}
      {\LL{p}\left([0, T] \times \Omega; [0, U]\right)}
      &
        \longrightarrow
      &
        \R
      \\
      u
      & \longmapsto
      & \displaystyle
        \!\!\int_0^T \!\!\!\int_{\Omega}
        \psi_{2,2}\left(\bar \rho_1(t, x),
        \bar \rho_2(t, x),
        \bar v(t, x)\right)  {u^p(t,x)} \dd x\, \dd t
    \end{array}
  \end{equation*}
  is convex and sequentially lower semicontinuous with respect to the
  strong topology, then it is also sequentially lower semicontinuous
  with respect to the weak topology and so
  \begin{align*}
    & \quad
      \liminf_{n \to +\infty} \int_0^T \int_{\Omega}
      \psi_{2,2}\left(\rho_1^n(t, x), \rho_2^n(t, x), v^n(t, x)\right)
      {u^p_n(t,x)}
      \dd x\, \dd t
    \\
    & \ge \liminf_{n \to +\infty} \int_0^T \int_{\Omega}
      \psi_{2,2}\left(\bar \rho_1(t, x),
      \bar \rho_2(t, x),
      \bar v(t, x)\right) {\bar u^p(t,x)} \dd x\, \dd t.
  \end{align*}
  This permits to conclude that
  \begin{equation*}
    \liminf_{n \to +\infty} J(u_n) \ge J(\bar u),
  \end{equation*}
  proving that $\bar u$ is an optimal control,
  i.e.~\eqref{eq:optimal-condition} holds. This concludes the proof.
\end{proof}

\begin{remark}
  \Cref{thm:existence-minimum-2} still holds under the following more
  general assumption on the function $\psi_{2}$ appearing in
  $J(u)$: 
  \begin{displaymath}
    \begin{array}{ccccc}
      \psi_2&:& \R^4& \to &\R
      \\
            && (\rho_1, \rho_2,v,u) & \mapsto & \psi_2 (\rho_1, \rho_2,v,u)
    \end{array}
  \end{displaymath}
  is Lipschitz continuous in all variables, positive and convex in
  $u$.  Indeed, since $(\rho_1^n, \rho_2^n, v^n)$ converges to
  $(\bar \rho_1, \bar \rho_2, \bar v)$, and the function $\psi_2$ is
  Lipschitz continuous, it holds
  \begin{align*}
    & \lim_{n \to +\infty} \int_0^T \int_{\Omega}
      \big[\psi_{2}\left(\rho_1^n(t, x), \rho_2^n(t, x), v^n(t, x), {u_n(t,x)}\right)
    \\
    & \qquad\qquad\qquad\qquad
      - \psi_{2,2}\left(\bar \rho_1(t, x),
      \bar \rho_2(t, x),
      \bar v(t, x), {u_n(t,x)}\right) \big]  \dd x\, \dd t =0.
  \end{align*}
  Therefore
  \begin{align*}
    & \quad
      \liminf_{n \to +\infty} \int_0^T \int_{\Omega}
      \psi_{2,2}\left(\rho_1^n(t, x), \rho_2^n(t, x), v^n(t, x),
      {u_n(t,x)} \right)
      \dd x\, \dd t
    \\
    & = \liminf_{n \to +\infty} \int_0^T \int_{\Omega}
      \psi_{2,2}\left(\bar \rho_1(t, x),
      \bar \rho_2(t, x),
      \bar v(t, x), {u_n(t,x)} \right)  \dd x\, \dd t.
  \end{align*}
  The convexity of $\psi_2$ with respect to $u$ ensures,
  by~\cite[Theorem~2.12]{zbMATH05703572} the weakly lower
  semicontinuity of the functional
  \begin{equation*}
    \begin{array}{rcc}
      {\LL{p}\left([0, T] \times \Omega; [0, U]\right)}
      &
        \longrightarrow
      &
        \R
      \\
      u
      & \longmapsto
      & \displaystyle
        \!\!\!\int_0^T \!\!\!\int_{\Omega}
        \psi_{2,2}\left(\bar \rho_1(t, x),
        \bar \rho_2(t, x),
        \bar v(t, x), {u(t,x)} \right) \dd x\, \dd t,
    \end{array}
  \end{equation*}
  so that
  \begin{align*}
    & \quad
      \liminf_{n \to +\infty} \int_0^T \int_{\Omega}
      \psi_{2,2}\left(\rho_1^n(t, x), \rho_2^n(t, x), v^n(t, x),
      {u_n(t,x)}\right)
      \dd x\, \dd t
    \\
    & \ge \liminf_{n \to +\infty} \int_0^T \int_{\Omega}
      \psi_{2,2}\left(\bar \rho_1(t, x),
      \bar \rho_2(t, x),
      \bar v(t, x), {\bar u(t,x)} \right)  \dd x\, \dd t,
  \end{align*}
  allowing to conclude that
  \begin{displaymath}
    \liminf_{n \to + \infty} J(u_n) \ge J(\bar u).
  \end{displaymath}
\end{remark}

\section{Necessary optimality conditions}
\label{sec:necessary-conditions}

In this section, given the initial data
$\rho_{1,o}, \rho_{2,o}, v_{o} \in \LL\infty\left(\Omega; \R_+\right)$,
we aim to prove
necessary conditions for optimal controls, i.e.~for controls
satisfying~\eqref{eq:optimal-condition} of
\Cref{thm:existence-minimum-2}. First we prove the
differentiability of the control-to-state map and then
we deduce optimality conditions using the adjoint system.

\subsection{Differentiability of the control-to-state map}

We follow the line
of~\cite{ContiGattiMiranville}, see also~\cite{Colli}.
Consider the control-to-state map $G$, defined as
\begin{equation}
  \label{eq:control-to-state}
  \begin{array}{llccc}
    G
    & : &{\LL2([0,T] \times \Omega ;[0,U])} & \to & \CC0 ([0,T];
                                                           \left(\LL2 (\Omega)\right)^3)
    \\
    && u & \mapsto & (\rho_1, \rho_2,v),
  \end{array}
\end{equation}
where the triple $(\rho_1, \rho_2,v)$ is the unique solution
to~\eqref{eq:PDE-model-3x3}-\eqref{eq:initial-condition}-\eqref{eq:boundary-condition},
corresponding to the control $u$ with initial data
$\rho_{1,o}, \rho_{2,o}, v_{o}$.

First, observe that, by~\Cref{thm:dependence-from-u-H1}, the map $G$
is Lipschitz continuous. Let now
$u^*\in {\LL2([0,T] \times \Omega ;[0,U])}$ be fixed and denote
the corresponding state by $G(u^*) = (\rho_1^*, \rho_2^*,v^*)$. Given
$\bar u \in {\LL2([0,T] \times \Omega ;[0,U])}$, we introduce
the linearized system at $ (\rho_1^*, \rho_2^*,v^*)$:
\begin{equation}
  \label{eq:L}
  \left\{
    \begin{array}{l}
      \pt X = \Delta X + \partial_{\rho_1}F_1(\rho_1^*, v^*) X + \partial_v F_1(\rho_1^*,v^*)Z
      \\
      \pt Y = \Delta Y + \partial_{\rho_1}F_2(\rho_1^*,\rho_2^*, v^*) X + \partial_{\rho_2}F_2(\rho_1^*,\rho_2^*, v^*) Y
      \\
      \qquad +  \partial_v F_2(\rho_1^*,\rho_2^*,v^*)Z
      \\
      \pt Z = \Delta Z + \partial_{\rho_1}F_3(\rho_1^*,\rho_2^*, v^*) X + \partial_{\rho_2}F_3(\rho_1^*,\rho_2^*, v^*) Y
      \\
      \qquad +  \partial_v F_3(\rho_1^*,\rho_2^*,v^*)Z+ \bar{u} - u^*,
    \end{array}
  \right.
\end{equation}
coupled with zero initial conditions and homogeneous Neumann boundary
conditions
\begin{equation}
  \label{eq:0-init-cond}
  \left\{
    \begin{array}{l}
      X(0, x) = 0
      \\
      Y(0, x) = 0
      \\
      Z(0,x) = 0
    \end{array}
  \right.
  \qquad
  \left\{
    \begin{array}{l}
      \partial_\nu X(t, \xi) = 0
      \\
      \partial_\nu Y(t, \xi) = 0
      \\
      \partial_\nu Z(t, \xi) = 0.
    \end{array}
  \right.
\end{equation}
Above, we used the following notation:
\begin{align}
  \nonumber
  F_1 (\rho_1, v) = \
  & (\alpha-\delta_1) \rho_1 - \beta \rho_1 v,
  \\
  \label{eq:F123}
  F_2 (\rho_1, \rho_2, v) =\
  & \beta \rho_1 v -\delta_2\rho_2,
  \\
  \nonumber
  F_3 (\rho_1, \rho_2, v) =\
  & b \delta_2 \rho_2 -B \rho_1 v - \delta_v v.
\end{align}

\begin{lemma}
  \label{lem:sol_lin}
  The system~\eqref{eq:L}--\eqref{eq:0-init-cond} has a unique strong solution $(X,Y,Z) \in \mathcal{X}^3$, satisfying
  \begin{equation}
    \label{eq:7}
    \norm{X}^2_{\mathcal{X}} + \norm{Y}^2_{\mathcal{X}} + \norm{Z}^2_{\mathcal{X}}
    \le
    C \norm{\bar u}^2_{\LL2 ((0,T)\times\Omega)},
  \end{equation}
  where $\mathcal{X}=\CC0 ([0,T]; \HH1 (\Omega))$.
\end{lemma}
\begin{proof}
  System~\eqref{eq:L}--\eqref{eq:0-init-cond} is linear parabolic, the
  coefficients of $X, Y$ and $Z$ are in
  $\LL\infty ([0,T]\times\Omega)$ by~\Cref{thm:global-existence}, the
  source term $\bar{u} - u^*$ is by hypothesis in
  $\LL2 ((0,T) \times \Omega)$, the initial data are zero, then
  smooth. Hence, by~\cite[Theorem~3.6]{MR4249451}
  or by~\cite[Theorem~1.1, Ch. IV]{Lions-1961}
  there exists a
  unique triple $(X,Y,Z) \in \mathcal{X}^3$ that
  solves~\eqref{eq:L}--\eqref{eq:0-init-cond} and
  satisfies~\eqref{eq:7}.
\end{proof}

Let $\lambda \in (0,1)$ and set
\begin{equation}
  \label{eq:8}
  u^\lambda = u^* + \lambda (\bar{u} - u^*).
\end{equation}
Clearly $u^\lambda \in \LL2 ((0,T)\times \Omega; [0,U])$, so that we can define the corresponding state
\begin{displaymath}
  (\rho_1^\lambda, \rho_2^\lambda, v^\lambda) =
  G(u^\lambda).
\end{displaymath}
Note that, as $\lambda \to 0$, by construction we have $u^\lambda \to u^*$ and by the Lipschitz continuity of the control-to-state map $G$ we have $(\rho_1^\lambda, \rho_2^\lambda, v^\lambda) \to (\rho_1^*, \rho_2^*, v^*)$.

The next proposition describes the directional derivative of the
control-to-state map $G$ at $u^*$.
\begin{prop}
  \label{prop:dir_der}
  The directional derivative of the control-to-state map $G$ in the direction $(\bar u - u^*)$ is given by
  \begin{equation}
    \label{eq:10}
    D_{(\bar u - u^*)} G (u^*) = (X,Y,Z),
  \end{equation}
  where the triple $(X,Y,Z)$ is the solution to the linearized
  system~\eqref{eq:L}--\eqref{eq:0-init-cond}.
\end{prop}
\begin{proof}
  Set
  \begin{align}
    \label{eq:9}
    X^\lambda = \
    & \frac{\rho_1^\lambda - \rho_1^*}{\lambda} - X,
    &
      Y^\lambda = \
    & \frac{\rho_2^\lambda - \rho_2^*}{\lambda} - Y,
    &
      Z^\lambda = \
    & \frac{v^\lambda - v^*}{\lambda} - Z.
  \end{align}
  We claim that the triple $(X^\lambda, Y^\lambda, Z^\lambda)$
  converges strongly to the point $(0,0,0)$ in
  $\CC0([0,T]; \LL2 (\Omega)^3) \cap \LL2 (0,T; \HH1 (\Omega)^3)$.

  Starting from the definition~\eqref{eq:9} 
  we write the system 
  \begin{equation*}
    \left\{
      \begin{array}{l@{}l}
        \pt X^\lambda =
        &\Delta X^\lambda
        + \partial_{\rho_1}F_1(a_1^\lambda, v^\lambda) X^\lambda
        + \partial_v F_1(\rho_1^*,c_1^\lambda) Z^\lambda
        + A_1 X + A_3 Z
        \\
        \pt Y^\lambda =
        &\Delta Y^\lambda
        + \partial_{\rho_1}F_2(a_2^\lambda,\rho_2^\lambda, v^\lambda) X^\lambda
        + \partial_{\rho_2}F_2(\rho_1^*,b_2^\lambda, v^\lambda) Y^\lambda
        \\
        & + \partial_v F_2(\rho_1^*,\rho_2^*,c_2^\lambda)Z^\lambda
        + B_1 X + B_2 Y + B_3 Z
        \\
        \pt Z^\lambda =
        & \Delta Z^\lambda
        + \partial_{\rho_1}F_3(a_3^\lambda,\rho_2^\lambda, v^\lambda) X^\lambda
        + \partial_{\rho_2}F_3(\rho_1^*,b_3^\lambda, v^\lambda) Y^\lambda
        \\
        & + \partial_v F_3(\rho_1^*,\rho_2^*,c_3^\lambda) Z^\lambda
        + C_1 X + C_2 Y + C_3 Z,
      \end{array}
    \right.
  \end{equation*}
  where $F_1, F_2, F_3$ are defined as in~\eqref{eq:F123},
  $a_1^\lambda, a_2^\lambda, a_3^\lambda $ are intermediate values
  between $\rho_1^\lambda$ and $\rho_1^*$, $b_2^\lambda, b_3^\lambda $
  are intermediate values between $\rho_2^\lambda$ and $\rho_2^*$,
  $c_1^\lambda, c_2^\lambda, c_3^\lambda $ are intermediate values
  between $v^\lambda$ and $v^*$, and
  \begin{align*}
    A_1 =\
    & \partial_{\rho_1}F_1(a_1^\lambda, v^\lambda)
      - \partial_{\rho_1}F_1(\rho_1^*, v^*)
      = \beta (v^*-v^\lambda),
    \\
    A_3 = \
    & \partial_{v}F_1(\rho_1^*, c_1^\lambda)
      - \partial_{v}F_1(\rho_1^*, v^*)
      = 0,
    \\
    B_1 = \
    & \partial_{\rho_1}F_2(a_2^\lambda,\rho_2^\lambda, v^\lambda)
      - \partial_{\rho_1}F_2(\rho_1^*,\rho_2^*, v^*)
      = \beta (v^\lambda-v^*),
    \\
    B_2 = \
    & \partial_{\rho_2}F_2(\rho_1^*,b_2^\lambda, v^\lambda)
      - \partial_{\rho_2}F_2(\rho_1^*,\rho_2^*, v^*)
      =0,
    \\
    B_3 = \
    & \partial_{v}F_2(\rho_1^*,\rho_2^*, c_2^\lambda)
      - \partial_{v}F_2(\rho_1^*,\rho_2^*, v^*)
      =0,
    \\
    C_1 = \
    & \partial_{\rho_1}F_3(a_3^\lambda,\rho_2^\lambda, v^\lambda)
      - \partial_{\rho_1}F_3(\rho_1^*,\rho_2^*, v^*)
      =B(v^*-v^\lambda),
    \\
    C_2 = \
    & \partial_{\rho_2}F_3(\rho_1^*,b_3^\lambda, v^\lambda)
      - \partial_{\rho_2}F_3(\rho_1^*,\rho_2^*, v^*)
      =0,
    \\
    C_3 = \
    & \partial_{v}F_3(\rho_1^*,\rho_2^*, c_3^\lambda)
      - \partial_{v}F_3(\rho_1^*,\rho_2^*, v^*)
      =0.
  \end{align*}
  We now test the equation for $X^\lambda$ (respectively $Y^\lambda$
  and $Z^\lambda$) with $X^\lambda$ (respectively $Y^\lambda$ and
  $Z^\lambda$). Exploiting the weakly coercivity of the bilinear
  forms, as in the proof of~\Cref{thm:dependence-from-u-H1}, since
  $\partial_v F_1(\rho_1^*,c_1^\lambda)=- \beta \rho_1^*$, we obtain:
  \begin{align*}
    & \frac12\frac{\dd{}}{\dd{t}}
      \norm{X^\lambda (t)}^2_{\LL2 (\Omega)}
      + \frac12 \norm{X^\lambda (t)}^2_{\HH1 (\Omega)}
    \\
    \le \
    & \int_\Omega \partial_v F_1(\rho_1^*,c_1^\lambda) X^\lambda Z^\lambda
      +\int_\Omega A_1 X X^\lambda
    \\
    &  + \left(\frac12+\delta_1 + \alpha
      + \beta \norm{v^\lambda (t)}_{\LL\infty (\Omega)}\right)
      \norm{X^\lambda(t)}^2_{\LL2 (\Omega)}
    \\
    \le \
    & \frac{\beta}{2} \norm{\rho_1^*(t)}^2_{\LL\infty (\Omega)}
      \norm{X^\lambda(t)}^2_{\LL2 (\Omega)}
      +\frac{\beta}{2} \norm{Z^\lambda(t)}^2_{\LL2 (\Omega)}
    \\
    & + \frac\beta2 \norm{v^* (t)-v^\lambda(t)}^2_{\LL2 (\Omega)}
      + \frac\beta2\norm{X (t)}^2_{\LL\infty (\Omega)} \norm{X^\lambda(t)}^2_{\LL2 (\Omega)}
    \\
    &+  \left(\frac12+\delta_1 + \alpha
      + \beta \norm{v^\lambda (t)}_{\LL\infty (\Omega)}\right)
      \norm{X^\lambda(t)}^2_{\LL2 (\Omega)},
  \end{align*}
  where the $\LL\infty$-norms appearing above can be controlled
  using~\eqref{eq:general-growth-est-rho1}--\eqref{eq:general-growth-est-v}
  and the fact that $X$ belongs to $\LL\infty ((0,T) \times \Omega)$
  by~\Cref{lem:sol_lin}.  Proceed similarly for $Y^\lambda$: since
  $\partial_{\rho_1} F_2(a_2^\lambda, \rho_2^\lambda, v^\lambda) =
  \beta v^\lambda$ and
  $\partial_{v} F_2(\rho_1^*, \rho_2^*,c_2^\lambda)=\beta \rho_1^*$,
  we get
  \begin{align*}
    & \frac12\frac{\dd{}}{\dd{t}}
      \norm{Y^\lambda (t)}^2_{\LL2 (\Omega)}
      + \frac12 \norm{Y^\lambda (t)}^2_{\HH1 (\Omega)}
    \\
    \le \
    & \int_\Omega \partial_{\rho_1} F_2(a_2^\lambda, \rho_2^\lambda, v^\lambda)
      X^\lambda Y^\lambda
      +\int_\Omega \partial_{v} F_2(\rho_1^*, \rho_2^*,c_2^\lambda)
      Y^\lambda Z^\lambda
      +\int_\Omega B_1 X Y^\lambda
    \\
    &  + \left(\frac12+\delta_2\right)
      \norm{Y^\lambda(t)}^2_{\LL2 (\Omega)}
    \\
    \le \
    & \frac{\beta}{2} \norm{v^\lambda(t)}^2_{\LL\infty (\Omega)}
      \norm{X^\lambda(t)}^2_{\LL2 (\Omega)}
      +\frac{\beta}{2} \norm{Y^\lambda(t)}^2_{\LL2 (\Omega)}
    \\
    & + \frac{\beta}{2} \norm{\rho_1^*(t)}^2_{\LL\infty (\Omega)}
      \norm{Y^\lambda(t)}^2_{\LL2 (\Omega)}
      +\frac{\beta}{2} \norm{Z^\lambda(t)}^2_{\LL2 (\Omega)}
    \\
    & + \frac\beta2 \norm{v^* (t)-v^\lambda(t)}^2_{\LL2 (\Omega)}
      + \frac\beta2\norm{X (t)}^2_{\LL\infty (\Omega)} \norm{Y^\lambda(t)}^2_{\LL2 (\Omega)}
    \\
    &+  \left(\frac12+\delta_2\right)
      \norm{Y^\lambda(t)}^2_{\LL2 (\Omega)},
  \end{align*}
  where the $\LL\infty$-norms appearing above can be controlled using~\eqref{eq:general-growth-est-rho1}--\eqref{eq:general-growth-est-v}
  and the fact that $X$ belongs to $\LL\infty ((0,T) \times \Omega)$
  by~\Cref{lem:sol_lin}. Lastly, consider $Z^\lambda$: since $\partial_{\rho_1} F_3(a_3^\lambda, \rho_2^\lambda, v^\lambda)=-B\, v^\lambda$ and $\partial_{\rho_2} F_2(\rho_1^*, b_3^\lambda,v^\lambda)=b\, \delta_2$, we get
\begin{align*}
    & \frac12\frac{\dd{}}{\dd{t}}
      \norm{Z^\lambda (t)}^2_{\LL2 (\Omega)}
      + \frac12 \norm{Z^\lambda (t)}^2_{\HH1 (\Omega)}
    \\
    \le \
    & \int_\Omega \partial_{\rho_1} F_3(a_3^\lambda, \rho_2^\lambda, v^\lambda)
      X^\lambda Z^\lambda
      +\int_\Omega \partial_{\rho_2} F_2(\rho_1^*, b_3^\lambda,v^\lambda)
      Y^\lambda Z^\lambda
      +\int_\Omega C_1 X Z^\lambda
    \\
    &  + \left(\frac12+\delta_v + B \norm{\rho_1^* (t)}_{\LL\infty (\Omega)}\right)
      \norm{Z^\lambda(t)}^2_{\LL2 (\Omega)}
    \\
    \le \
    & \frac{B}{2} \norm{v^\lambda(t)}^2_{\LL\infty (\Omega)}
      \norm{X^\lambda(t)}^2_{\LL2 (\Omega)}
      +\frac{B}{2} \norm{Z^\lambda(t)}^2_{\LL2 (\Omega)}
    \\
    & + \frac{b \,\delta_2}{2}
      \norm{Y^\lambda(t)}^2_{\LL2 (\Omega)}
      +\frac{b \, \delta_2}{2} \norm{Z^\lambda(t)}^2_{\LL2 (\Omega)}
    \\
    & + \frac{B}2 \norm{v^* (t)-v^\lambda(t)}^2_{\LL2 (\Omega)}
      + \frac{B}2\norm{X (t)}^2_{\LL\infty (\Omega)} \norm{Z^\lambda(t)}^2_{\LL2 (\Omega)}
    \\
    &+  \left(\frac12+\delta_v + B \norm{\rho_1^* (t)}_{\LL\infty (\Omega)}\right)
      \norm{Z^\lambda(t)}^2_{\LL2 (\Omega)},
  \end{align*}
  where the $\LL\infty$-norms appearing above can be controlled using~\eqref{eq:general-growth-est-rho1}--\eqref{eq:general-growth-est-v}
  and the fact that $X$ belongs to $\LL\infty ((0,T) \times \Omega)$
  by~\Cref{lem:sol_lin}.

  Collecting together the estimates above, for a.e.~$t \in [0,T]$ we obtain
  \begin{align*}
    & \frac{\dd{}}{\dd{t}}\left(
      \norm{X^\lambda (t)}^2_{\LL2 (\Omega)}+\norm{Y^\lambda (t)}^2_{\LL2 (\Omega)}+
      \norm{Z^\lambda (t)}^2_{\LL2 (\Omega)}\right)
    \\
    & + \norm{X^\lambda (t)}^2_{\HH1 (\Omega)}+\norm{Y^\lambda (t)}^2_{\HH1 (\Omega)}+
    \norm{Z^\lambda (t)}^2_{\HH1 (\Omega)}
    \\
    \le \
    & C(t) \left(\norm{X^\lambda (t)}^2_{\LL2 (\Omega)}+\norm{Y^\lambda (t)}^2_{\LL2 (\Omega)}+
      \norm{Z^\lambda (t)}^2_{\LL2 (\Omega)}\right)
    \\
    & + (B + 2 \beta)  \norm{v^* (t)-v^\lambda(t)}^2_{\LL2 (\Omega)},
  \end{align*}
  where we set
  \begin{align*}
    C(t) = \
    \max\Bigl\{
    &
    \beta \norm{\rho_1^* (t)}^2_{\LL\infty (\Omega)}
    +  \beta \norm{X(t)}^2_{\LL\infty (\Omega)}
      + 1+2\delta_1+2\alpha
    \\
    & \quad + 2  \beta \norm{v^\lambda (t)}_{\LL\infty (\Omega)}
     + \left(\beta+B\right) \norm{v^\lambda (t)}^2_{\LL\infty (\Omega)},
    \\
    & \beta + \beta \norm{\rho_1^* (t)}^2_{\LL\infty (\Omega)}
      + \beta \norm{X (t)}^2_{\LL\infty (\Omega)}
      + 1 + 2 \, \delta_2 + b \, \delta_2,
    \\
    & 2\, \beta + B + b \, \delta_2
      +B \norm{X (t)}^2_{\LL\infty (\Omega)}
      +1+2\, \delta_v +2 B \norm{\rho_1^* (t)}^2_{\LL\infty (\Omega)}
      \Bigr\}.
  \end{align*}
  An application of Gronwall's inequality yields
  \begin{equation}
    \label{eq:12}
    \begin{aligned}
      & \norm{X^\lambda (t)}^2_{\LL2 (\Omega)}+\norm{Y^\lambda
        (t)}^2_{\LL2 (\Omega)}+ \norm{Z^\lambda (t)}^2_{\LL2 (\Omega)}
      \\
      & + \int_0^t\left(
        \norm{X^\lambda (s)}^2_{\HH1(\Omega)}
        +\norm{Y^\lambda (s)}^2_{\HH1 (\Omega)}
        +\norm{Z^\lambda (s)}^2_{\HH1 (\Omega)} \right) \dd{s}
      \\
      \le \ & \int_0^t (B + 2 \beta) \norm{v^*
        (t)-v^\lambda(s)}^2_{\LL2 (\Omega)} \exp\left(\int_s^t
        C(\tau)\dd\tau\right)\dd{s}.
    \end{aligned}
  \end{equation}
  Since, in the limit $\lambda \to 0$, we have the convergence
  $v^\lambda \to v^*$, by~\eqref{eq:12} we obtain the thesis, i.e.
  \begin{displaymath}
    (X^\lambda, Y^\lambda, Z^\lambda) \to (0,0,0)
    \text{ strongly in }
    \CC0 ([0,T]; \LL2 (\Omega)^3) \cap \LL2 (0,T; \HH1 (\Omega)^3).
  \end{displaymath}
  Due to the definition~\eqref{eq:9} of the triple
  $(X^\lambda, Y^\lambda, Z^\lambda)$, this amounts to
  \begin{displaymath}
    \left(
      \frac{\rho_1^\lambda-\rho_1^*}{\lambda},
      \frac{\rho_2^\lambda-\rho_2^*}{\lambda},
      \frac{v^\lambda-v^*}{\lambda}
    \right)
    \overset{\lambda \to 0}{\longrightarrow}
    \left(X,Y,Z\right),
  \end{displaymath}
  which in terms of the control-to-state operator $G$ gives its
  directional derivative in the direction $\bar u - u^*$:
  \begin{displaymath}
    D_{(\bar u - u^*)} G(u^*) =
    (X,Y,Z),
  \end{displaymath}
  concluding the proof.
\end{proof}

\subsection{Adjoint system and necessary conditions}

Let us now consider the functional $J$ defined
in~\eqref{eq:funct-J}. Observe that $J$ is actually a function also of
$(\rho_1, \rho_2, v)$, and not only of the control $u$, thus it would
be more precise to write $ J (\rho_1, \rho_2, v, u)$. The
control-to-state operator $G$ introduced in~\eqref{eq:control-to-state}
allows to write $(\rho_1, \rho_2, v) = G(u)$, so that we can define the reduced cost functional $f$ as
\begin{equation}
  \label{eq:f}
  f(u) := J (\rho_1, \rho_2, v, u) = J\left(G(u),u\right).
\end{equation}
We introduce the following assumptions on the cost functions
$\psi_1$ and $\psi_2$.
\begin{enumerate}[label=($\boldsymbol\psi$)]
    \item \label{hyp:psi} $\psi_1$ and $\psi_2$ are $\CC1$ functions.
    Moreover, for every $M>0$, there exists $L_M > 0$ such
    that
    \begin{align*}
  \modulo{\nabla \psi_1(\bar \rho_1, \bar \rho_2, \bar v)
  -
  \nabla \psi_1(\hat \rho_1, \hat \rho_2, \hat v)}
  \le \
  & L_M \modulo{(\bar \rho_1, \bar \rho_2, \bar v)-
    (\hat \rho_1, \hat \rho_2, \hat v)},
  \\
  \modulo{\nabla \psi_2(\bar \rho_1, \bar \rho_2, \bar v, \bar u)
  \!-\!
  \nabla \psi_2(\hat \rho_1, \hat \rho_2, \hat v, \hat u)}
  \le \
  & L_M \modulo{(\bar \rho_1, \bar \rho_2, \bar v, \bar u)\!-\!
    (\hat \rho_1, \hat \rho_2, \hat v, \hat u)},
\end{align*}
for every
$\bar \rho_1, \bar \rho_2, \bar v, \bar u, \hat \rho_1, \hat \rho_2,
\hat v, \hat u \in [0,M]$.
\end{enumerate}

Thanks to~\cite[Lemma~4.12]{zbMATH05703572},
the functional $J$ admits partial derivatives,
while~\Cref{prop:dir_der} ensures the differentiability of the
control-to-state operator $G$. Hence, the reduced cost
functional $f$~\eqref{eq:f} is differentiable in
$\LL\infty ((0,T) \times\Omega)$.

Consider a set of admissible control $\mathcal U \neq \emptyset$ that is a
closed convex subset of $\LL2 ((0,T)\times \Omega; [0,U])$.
Let $u^* \in \mathcal U$ be a locally optimal control for
problem~\eqref{eq:PDE-model-3x3}--\eqref{eq:initial-condition}--\eqref{eq:boundary-condition}
subject to the minimization of the functional
$J$~\eqref{eq:funct-J}. Then, for any $\bar u \in \mathcal U$, defining
$u^\lambda$ as in~\eqref{eq:8} for $\lambda \in (0,1)$, the following
inequality holds
\begin{equation*}
  f (u^\lambda) - f(u^*) \ge 0.
\end{equation*}
Dividing by $\lambda$ and passing to the limit as $\lambda \to 0$, we
obtain
\begin{equation}
  \label{eq:14}
  f'(u^*) (\bar u - u^*) \ge 0
  \quad \forall \bar u \in \mathcal U.
\end{equation}
Due to the definition of $f$~\eqref{eq:f}, using the chain rule
and~\eqref{eq:10}, we can compute $f'$ appearing in~\eqref{eq:14}: for
any $\bar u \in \mathcal U$
\begin{align}
  \nonumber
  0 \le \
  & f'(u^*) (\bar u- u^*)
  \\ \nonumber
  = \
  & \nabla_{(\rho_1, \rho_2, v)} J\left(G(u^*), u^*\right)
    \cdot D_{(\bar u -u^*)}G (u^*)
    + \partial_u J\left(G(u^*), u^*\right) (\bar u - u^*)
  \\
  \nonumber
  = \
  & \partial_{\rho_1} J\left(G(u^*), u^*\right) X
    + \partial_{\rho_2} J\left(G(u^*), u^*\right) Y
    + \partial_{v} J\left(G(u^*), u^*\right) Z
  \\
  \nonumber
  & + \partial_{u} J\left(G(u^*), u^*\right)  (\bar u - u^*)
  \\
  \nonumber
  = \
  & \int_\Omega
    \partial_{\rho_1}
    \psi_1\left(\rho_1^*(T, x), \rho_2^*(T, x), v^*(T, x)\right)
    X (T,x)
    \dd x
  \\
  \nonumber
  & + \int_\Omega
    \partial_{\rho_2}
    \psi_1\left(\rho_1^*(T, x), \rho_2^*(T, x), v^*(T, x)\right)
    Y (T,x)
    \dd x
  \\
  \label{eq:15}
  & + \int_\Omega
    \partial_{v}
    \psi_1\left(\rho_1^*(T, x), \rho_2^*(T, x), v^*(T, x)\right)
    Z(T,x)
    \dd x
  \\
  \nonumber
  & + \int_0^T \int_\Omega
    \partial_{\rho_1} \psi_2 \left(\rho_1^*(t, x), \rho_2^*(t, x), v^*(t, x), u^* (t,x)\right)
    X (t,x)
    \dd x \dd t
  \\
  \nonumber
  & + \int_0^T \int_\Omega
    \partial_{\rho_2} \psi_2 \left(\rho_1^*(t, x), \rho_2^*(t, x), v^*(t, x), u^* (t,x)\right)
    Y(t,x)
    \dd x \dd t
  \\
  \nonumber
  & + \int_0^T \int_\Omega
    \partial_{v} \psi_2 \left(\rho_1^*(t, x), \rho_2^*(t, x), v^*(t, x), u^* (t,x)\right)
    Z(t,x)
    \dd x \dd t
  \\
  \nonumber
  &  + \int_0^T \int_\Omega
    \partial_u \psi_2 \left(\rho_1^*(t, x), \rho_2^*(t, x), v^*(t, x), u^* (t,x)\right)
    \left(\bar u (t,x) - u^*(t,x)\right)
    \dd x \dd t,
\end{align}
where the triple $(X,Y,Z)$ is the solution to the linearized
system~\eqref{eq:L}--\eqref{eq:0-init-cond}.

Introduce the adjoint system, in the variables $w, y, z$:
\begin{equation}
  \label{eq:adjoint}
  \left\{
    \begin{array}{l}
      -\pt w - \Delta w = \partial_{\rho_1} F_1 (\rho_1^*, v^*) w
      + \partial_{v} F_1 (\rho_1^*, v^*) z
      + \partial_{\rho_1} \psi_2 (\rho_1^*, \rho_2^*, v^*, u^*)
      \\
      -\pt y - \Delta y =
      \partial_{\rho_1} F_2 (\rho_1^*, \rho_2^*, v^*) w
      + \partial_{\rho_2} F_2 (\rho_1^*, \rho_2^*, v^*) y
    \\
      \qquad + \partial_{v} F_2 (\rho_1^*, \rho_2^*, v^*) z
      + \partial_{\rho_2} \psi_2 (\rho_1^*, \rho_2^*, v^*, u^*)
      \\
      -\pt z - \Delta z =
      \partial_{\rho_1} F_3 (\rho_1^*, \rho_2^*, v^*) w
      + \partial_{\rho_2} F_3 (\rho_1^*, \rho_2^*, v^*) y
      \\
      \qquad
      + \partial_{v} F_3 (\rho_1^*, \rho_2^*, v^*) z
      + \partial_{\rho_2} \psi_2 (\rho_1^*, \rho_2^*, v^*, u^*)
    \end{array}
  \right.
\end{equation}
with the following initial and boundary conditions,
\begin{equation}
  \label{eq:init_bound_conditions_adjoint}
  \left\{
    \begin{array}{l}
      w(T,x) = \partial_{\rho_1} \psi_1 (\rho_1^*(T,x), \rho_2^* (T,x), v^* (T,x))
      \\
      y(T,x) = \partial_{\rho_2} \psi_1 (\rho_1^*(T,x), \rho_2^* (T,x), v^* (T,x))
      \\
      z(T,x) = \partial_{v} \psi_1 (\rho_1^*(T,x), \rho_2^* (T,x), v^* (T,x))
    \end{array}
  \right.
  \qquad
  \left\{
    \begin{array}{l}
      \partial_\nu w(t, \xi) = 0
      \\
      \partial_\nu y(t, \xi) = 0
      \\
      \partial_\nu z(t, \xi) = 0.
    \end{array}
  \right.
\end{equation}
A result similar to~\Cref{lem:sol_lin} holds, implying that there
exists a unique solution
to~\eqref{eq:adjoint}--\eqref{eq:init_bound_conditions_adjoint}


\begin{lemma}{\cite[Theorem~3.18]{zbMATH05703572}}
    Assume~\ref{hyp:psi} holds.
  Let $(X,Y,Z)$ be the solution to the linearized
  problem~\eqref{eq:L}--\eqref{eq:0-init-cond}. Let $(w,y,z)$ be the
  weak solution to the adjoint
  problem~\eqref{eq:adjoint}--\eqref{eq:init_bound_conditions_adjoint}. Then,
  \begin{align}
    \nonumber
    & \int_\Omega
    \partial_{\rho_1} \psi_1 (\rho_1^*(T,x), \rho_2^* (T,x), v^*
    (T,x)) X(T,x) \dd{x}
    \\
    \nonumber
    & + \int_\Omega
    \partial_{\rho_2} \psi_1 (\rho_1^*(T,x), \rho_2^* (T,x), v^*
    (T,x)) Y(T,x) \dd{x}
    \\
    \nonumber
    & + \int_\Omega
    \partial_{v} \psi_1 (\rho_1^*(T,x), \rho_2^* (T,x), v^* (T,x))
    Z(T,x) \dd{x}
    \\
    \nonumber
    & + \int_0^T\int_\Omega
    \partial_{\rho_1} \psi_2 (\rho_1^*(t,x), \rho_2^*(t,x), v^*(t,x),
    u^*(t,x)) X(t,x) \dd{x}\dd{t}
    \\
    \nonumber
    & + \int_0^T\int_\Omega
    \partial_{\rho_2} \psi_2 (\rho_1^*(t,x), \rho_2^*(t,x), v^*(t,x),
    u^*(t,x))Y(t,x) \dd{x}\dd{t}
    \\
    \nonumber
    & + \int_0^T\int_\Omega
    \partial_{\rho_2} \psi_2 (\rho_1^*(t,x), \rho_2^*(t,x), v^*(t,x),
    u^*(t,x)) Z(t,x) \dd{x}\dd{t}
    \\
    \label{eq:16}
    =
    & \int_0^T\int_\Omega (\bar u(t,x) -u^*(t,x)) z(t,x) \dd{x}\dd{t}.
  \end{align}
\end{lemma}

\begin{theorem}
  Let $u^* \in \mathcal U$ be a locally optimal control for
  problem~\eqref{eq:PDE-model-3x3}--\eqref{eq:initial-condition}--\eqref{eq:boundary-condition}
  subject to the minimization of the functional
  $J$~\eqref{eq:funct-J}. If $(w,y,z)$ is the associated state solving
  problem~\eqref{eq:adjoint}--\eqref{eq:init_bound_conditions_adjoint},
  then the following variational inequality holds for all
  $\bar u \in \mathcal U$
  \begin{equation}
    \label{eq:19}
    \int_0^T \int_\Omega
    \bigl( z + \partial_u \psi_2\left(\rho_1^*, \rho_2^*, v^*, u^*\right) \bigr) (\bar u -u^*)
      \dd x \dd t \ge 0.
    \end{equation}
  \end{theorem}

  \begin{proof}
  Inserting~\eqref{eq:16} into~\eqref{eq:15} leads to
\begin{displaymath}
  \begin{split}
    & \int_0^T\int_\Omega (\bar u(t,x) -u^*(t,x)) z(t,x) \dd{x}\dd{t}
    \\
    & +
    \int_0^T \!\!\int_\Omega
    \partial_u\psi_2 \left(\rho_1^*(t, x), \rho_2^*(t, x), v^*(t, x), u^* (t,x)\right)
    \left(\bar u (t,x) \!- \!u^*(t,x)\right)
    \dd x \dd t \ge 0,
  \end{split}
\end{displaymath}
concluding the proof.
\end{proof}

\begin{remark}
  \label{rmk:functionals-conditions}
  With reference to the functionals introduced
  in~\Cref{rmk:functionals}, we deduce the following necessary
  conditions.

  If
  \begin{equation*}
    J(u) = \gamma_1 \int_\Omega \rho_1(T, x) \dd x
    + \gamma_2 \int_\Omega \rho_2(T, x) \dd x,
  \end{equation*}
  for suitable $\gamma_1, \gamma_2 \ge 0$, then $\psi_2 = 0$ and
  so~\eqref{eq:19} becomes
  \begin{equation*}
    \int_0^T \int_\Omega
    z  (\bar u -u^*)
    \dd x \dd t \ge 0.
  \end{equation*}

  If
  \begin{equation*}
    J(u) = \gamma_1 \int_\Omega \rho_1(T, x) \dd x
    + \gamma_2 \int_\Omega \rho_2(T, x) \dd x
    + {\int_0^T \int_\Omega u^p(t,x) \dd x \dd t},
  \end{equation*}
  or
  \begin{equation*}
    \begin{split}
      J(u) & = \int_\Omega \left(\rho_1(T, x) - \bar \rho(x)\right)^2
      \dd x + \int_0^T \int_\Omega \left(\rho_1(t, x) - \bar
        \rho(x)\right)^2 \dd x \dd t
      \\
      & \quad + {\int_0^T\int_\Omega u^p(t,x)\dd x \dd t}.
    \end{split}
  \end{equation*}
  where $\gamma_1, \gamma_2 \ge 0$, $\bar \rho \in \LL2(\Omega)$, and
  $p \ge 1$, then $\partial_u \psi_2 = p u^{p-1}$ and so~\eqref{eq:19}
  becomes
  \begin{equation*}
    \int_0^T \int_\Omega
    (z + p (u^*)^{p-1} ) (\bar u -u^*)
    \dd x \dd t \ge 0.
  \end{equation*}
\end{remark}

\appendix

\section{Preliminary results on the linear parabolic equation
  \texorpdfstring{$\partial_t u = \Delta u + c(t,x) \,u + f(t,x)$}{}}
\label{sec:base}

Let $\Omega \subseteq \R^n$ be a bounded domain with Lipschitz
boundary $\partial \Omega$, fix $T>0$ and set
$\Omega_T = (0,T) \times \Omega$ and
$S_T = (0,T) \times \partial \Omega$. Consider the following problem
\begin{equation}
  \label{eq:parBase}
  \left\{
    \begin{array}{ll}
      \partial_t u - \Delta u + c(t,x) \, u = f(t,x)
      & \mbox{ in } \Omega_T,  \\
      u(0,x) = g(x)
      & \mbox{ in } \Omega,\\
      \partial_\nu u(t,\xi) = 0
      & \mbox{ on } S_T,
    \end{array}
  \right.
\end{equation}
where $\nu$ is the outward normal on $\Omega$ at the boundary
$\partial \Omega$, which exists for $\mathcal H^{n-1}$-a.e.
$\xi \in \partial \Omega$.  Assume that $f \in \LL2(\Omega_T; \R)$,
$g \in \LL2(\Omega; \R)$, and $c \in \LL\infty(\Omega_T; \R)$. Fix
$c_o \ge 1$ such that $\norma{c}_{\LL\infty(\Omega_T)} \leq c_o$.  We
define the Hilbert space
\begin{equation}
  \label{eq:hilbert-space}
  \HH1\!(0,T; \HH1\!(\Omega), \HH1\!(\Omega)^*)
  \!=\! \left\{u \!\in\! \LL2\!\left(0, T; \HH1\!\left(\Omega\right)\right)\!\!:
    \dot u \!\in\! \LL2\!\left(0, T; \HH1\!\left(\Omega\right)^*\right)\right\}
\end{equation}
endowed with inner product
\begin{equation*}
  \left(u_1, u_2\right)_{\HH1(0,T; \HH1(\Omega), \HH1(\Omega)^*)}
  = \int_0^T \left(u_1(t), u_2(t)\right)_{\HH1}  \dd t
  + \int_0^T \!\!\left(\dot u_1(t), \dot u_2(t)\right)_{\HH1^*}  \dd t
\end{equation*}
and norm
\begin{equation*}
  \norm{u}_{\HH1(0,T; \HH1(\Omega), \HH1(\Omega)^*)}^2
  = \int_0^T \norm{u(t)}_{\HH1}^2 \dd t
  + \int_0^T \norm{\dot u(t)}_{\HH1^*}^2 \dd t.
\end{equation*}

Following~\cite[Chapter~10]{Salsa}, we introduce the definition of
weak solution of problem~\eqref{eq:parBase}.
\begin{definition}
  \label{def:sol-linear-equation}
  A function $u \in \HH1(0,T; \HH1(\Omega), \HH1(\Omega)^*)$ is a weak
  solution to~\eqref{eq:parBase} if $u(0) = g$ and
  \begin{displaymath}
    \langle \dot{u}(t), v\rangle_* + B\left(u(t), v; t\right)
    = \langle f(t),v \rangle_*
  \end{displaymath}
  for all $v \in \HH1(\Omega)$ and for a.e.~$t \in (0,T)$, where
  \begin{equation}
    \label{eq:bilinear}
    B(u,v;t) = \int_\Omega \left[
      \nabla u \cdot \nabla v + c(t,x) \, u \, v \right]\dd x,
  \end{equation}
  and $\langle \cdot, \cdot\rangle_*$ denotes the duality between
  $\HH1(\Omega)^*$ and $\HH1(\Omega)$.
\end{definition}
The bilinear form $B$~\eqref{eq:bilinear} is continuous, with
\begin{displaymath}
  \modulo{B(u,v;t)} \leq (1+c_o) \norma{u}_{\HH1(\Omega)} \norma{v}_{\HH1(\Omega)}.
\end{displaymath}
Moreover, $B$ is weakly coercive, since, for every $\lambda > c_o$ and
$\alpha \in \mathopen]0,1]$,
\begin{displaymath}
  B(u,u;t) + \lambda \norma{u}_{\LL2(\Omega)}^2 \geq
  \alpha \norma{u}_{\HH1(\Omega)}^2.
\end{displaymath}
Finally, for every $u,v \in \HH1(\Omega)$, the map
$t \mapsto B(u,v;t)$ is measurable by Fubini's theorem. Hence, we can
apply~\cite[Theorem~10.6]{Salsa}:
\begin{theorem}
  \label{thm:parBase}
  There exists a unique weak solution $u$ to
  problem~\eqref{eq:parBase} in the sense
  of~\Cref{def:sol-linear-equation}.  Moreover, for every
  $t \in [0, T]$,
  \begin{align}
    \label{eq:L2-estimate}
    \norma{u(t)}_{\LL2(\Omega)}^2
    & \leq
      e^{2 \, c_o \, t}
      \left\{
      \norma{g}_{\LL2(\Omega)}^2
      +
      \int_0^t \norma{f(s)}^2_{\HH1(\Omega)^*} \dd s
      \right\},
    \\
    \nonumber
    \int_0^t \norma{u(s)}_{\HH1(\Omega)}^2 \dd s
    &
      \leq
      e^{2 \, c_o \, t}
      \left\{
      \norma{g}_{\LL2(\Omega)}^2
      +
      \int_0^t \norma{f(s)}^2_{\HH1(\Omega)^*} \dd s
      \right\},
    \\
    \nonumber
    \int_0^t \norm{\dot u(s)}^2_{\HH1\left(\Omega\right)^*} \dd s
    & \le 2 \left(1 + c_o\right)^2 e^{2 c_o t} \norm{g}^2_{\LL2\left(\Omega\right)}
    \\
    \nonumber
    & \quad
      + \left(2 \left(1 + c_o\right)^2 e^{2 c_o t} + 2\right)
      \int_0^t \norm{f(s)}^2_{\HH1\left(\Omega\right)^*} \dd s.
  \end{align}
\end{theorem}

A first simple consequence is the continuity of the solution operator
for~\eqref{eq:parBase}.
\begin{corollary}
  \label{cor:solution-operator}
  The operator, which associates to every
  $f \in \LL2\left(\Omega_T; \R\right)$ and
  $g \in \LL2\left(\Omega; \R\right)$ the unique solution
  to~\eqref{eq:parBase}, is linear and continuous as a map
  \begin{equation*}
    \LL2\left(\Omega_T; \R\right) \times \LL2\left(\Omega; \R\right)
    \longrightarrow \HH1\left(0, T; \HH1\left(\Omega\right),
      \HH1\left(\Omega\right)^*\right).
  \end{equation*}
\end{corollary}

A second consequence of the estimates provided by~\Cref{thm:parBase}
is the stability of solutions to problem~\eqref{eq:parBase} with
respect to the source function $f$.
\begin{prop}
  \label{prop:stabBase}
  Let $u_1$ and $u_2$ solve
  \begin{displaymath}
    \left\{
      \begin{array}{l}
        \partial_t u_1 \!-\! \Delta u_1 + c(t,x) \, u_1 = f_1(t,x),  \\
        u_1(0,x) = g(x),\\
        \partial_\nu u_1(t,\xi) = 0,
      \end{array}
    \right.
    \left\{
      \begin{array}{l}
        \partial_t u_2 - \Delta u_2 + c(t,x) \, u_2 = f_2(t,x),  \\
        u_2(0,x) = g(x),\\
        \partial_\nu u_2(t,\xi) = 0,
      \end{array}
    \right.
  \end{displaymath}
  with $g \in \LL2(\Omega)$, $c \in \LL\infty (\Omega_T)$ and
  $f_1, f_2 \in \LL2(\Omega_T)$. Then
  \begin{equation*}
    \norma{u_1(t) - u_2(t)}_{\LL2(\Omega)}^2
    \leq
    e^{2 \, c_o \, t}
    \int_0^t \norma{f_1(s) - f_2(s)}_{\HH1(\Omega)^*}^2 \dd{s}.
  \end{equation*}
\end{prop}
\begin{proof}
  Set $u:= u_1 -u_2$. Clearly, $u$ solves
  \begin{displaymath}
    \left\{
      \begin{array}{l}
        \partial_t u - \Delta u + c(t,x) \, u = f_1(t,x) - f_2(t,x),  \\
        u(0,x) = 0,\\
        \partial_\nu u(t,\xi) = 0.
      \end{array}
    \right.
  \end{displaymath}
  Applying Theorem~\ref{thm:parBase} to $u$ we obtain the thesis.
\end{proof}

The following a-priori $\LL\infty$ estimate holds.

\begin{prop}
  \label{prop:a-priori-growth}
  Let $g \in \LL\infty\left(\Omega; \R_+\right)$,
  $f \in \LL\infty\left(\Omega_T; \R_+\right)$, and
  $c \in \LL\infty\left(\Omega_T; \R\right)$.  Let $u$ be the unique
  weak solution to~\eqref{eq:parBase}.  Then, for $t \in [0, T]$,
  \begin{equation*}
    0 \le u(t, x) \le
    \left\{
      \begin{array}{l@{\quad}l}
        \left(\norm{g}_{\LL\infty} +
        \frac{\norm{f}_{\LL\infty}}{\norma{c}_{\LL\infty}}\right)
        e^{\norm{c}_{\LL\infty} t} - \frac{\norm{f}_{\LL\infty}}
        {\norma{c}_{\LL\infty}},
        &
          \norm{c}_{\LL\infty} > 0,
          \vspace{.2cm}\\
        \norm{g}_{\LL\infty} + \norm{f}_{\LL\infty} t,
        &
          \norm{c}_{\LL\infty} = 0.
      \end{array}
    \right.
  \end{equation*}
\end{prop}

\begin{proof}
  Define the linear operator
  \begin{equation*}
    v \mapsto \mathcal Pv := \partial_t v - \Delta v - c v.
  \end{equation*}
  Clearly $\mathcal Pu = f \ge 0$ and $u(0) = g \ge 0$; hence
  $u(t, x) \ge 0$ by the weak maximum principle; see for
  example~\cite[Theorem~10.18 and Remark~10.19]{Salsa}.

  Define
  \begin{equation*}
    w(t, x) = \begin{cases}
      \left(\norm{g}_{\LL\infty} +
        \dfrac{\norm{f}_{\LL\infty}}{\norma{c}_{\LL\infty}}\right)
      e^{\norm{c}_{\LL\infty} t} - \dfrac{\norm{f}_{\LL\infty}}{\norma{c}_{\LL\infty}}
      & \mbox{ if } c \neq 0,
      \\[8pt]
      \norm{g}_{\LL\infty} + \norm{f}_{\LL\infty} \, t
      & \mbox{ if } c \equiv 0.
    \end{cases}
  \end{equation*}
  In the case $c\neq 0$, we have
  \begin{align*}
    \mathcal P\left(w - u\right)
    = \
    &  \mathcal Pw - \mathcal Pu
    \\
    =  \
    &  \norm{c}_{\LL\infty} \left(\norm{g}_{\LL\infty} +
      \frac{\norm{f}_{\LL\infty}}{\norma{c}_{\LL\infty}}\right)
      e^{\norm{c}_{\LL\infty} t}
    \\
    & \, - c(t,x) \left(\norm{g}_{\LL\infty} +
      \frac{\norm{f}_{\LL\infty}}{\norma{c}_{\LL\infty}}\right)
      e^{\norm{c}_{\LL\infty} t}
      + c(t,x) \frac{\norm{f}_{\LL\infty}}{\norma{c}_{\LL\infty}} - f(t,x)
    \\
    = &
        \left(\norma{c}_{\LL\infty} - c(t,x) \right)
        \norm{g}_{\LL\infty} e^{\norm{c}_{\LL\infty} t}
        + \norm{f}_{\LL\infty}e^{\norm{c}_{\LL\infty} t}
    \\
    & - \frac{c(t,x)}{\norma{c}_{\LL\infty}} \norm{f}_{\LL\infty}
      \left(e^{\norm{c}_{\LL\infty} t} -1\right) - f(t,x)
    \\
    \ge\
    &\norma{f}_{\LL\infty} \left(e^{\norm{c}_{\LL\infty} t} -1\right)
      \left( 1 - \frac{c(t,x)}{\norma{c}_{\LL\infty}}\right)
    \\
    \ge \
    & 0.
  \end{align*}
  On the other hand, in the case $c \equiv 0$ we have
  \begin{displaymath}
    \mathcal P\left(w - u\right)
    =  \norma{f}_{\LL\infty} - f(t,x) \ge 0.
  \end{displaymath}
  In both cases, $\partial_\nu\left(w - u\right) = 0$ in $S_T$ and
  \begin{equation*}
    \left(w(0, x) - u(0, x)\right) = \norm{g}_{\LL\infty} - g(x)
    \ge 0.
  \end{equation*}
  Hence, the weak maximum principle implies that $w \ge u$ in
  $\Omega_T$, completing the proof.
\end{proof}

We briefly recall a regularity result, see~\cite[Remark~10.17]{Salsa}:
the more regular the initial data, the more regular the solution.

\begin{prop}
  \label{prop:regularity}
  Let $u$ be the unique weak solution to problem~\eqref{eq:parBase} in
  the sense of~\Cref{def:sol-linear-equation}.
  If $g \in \HH1(\Omega)$, $f \in \LL2(0,T; \LL2(\Omega))$, and
  $c \in \LL\infty(\Omega_T)$, then
  $u\in \LL\infty(0,T; \HH1(\Omega))$ and
  $\dot u \in \LL2(0,T; \LL2(\Omega))$.\\
  If, in addition, $\Omega$ is a $\CC2$--domain,
  then $u \in \LL2(0,T; \HH2(\Omega))$.\\
\end{prop}

\begin{proof}
  The proof is based on the Faedo-Galerkin approximation for
  problem~\eqref{eq:parBase}; see~\cite[Theorem~10.14]{Salsa} for a
  similar case.

  Take a sequence $w_s$ of eigenvalues of the Laplace operator in
  $\Omega$ with $0$ Neumann boundary condition. We select the
  eigenvalues such that the closure of their span coincides with the
  space $\HH1\left(\Omega\right)$, they are orthogonal in
  $\HH1\left(\Omega\right)$ and orthonormal in
  $\LL2\left(\Omega\right)$.  For every $m > 1$, define a
  Faedo-Galerkin approximation $u_m$ of $u$ as
  \begin{equation*}
    u_m(t) = \sum_{j = 1}^{m} c_{jm}(t) w_j,
  \end{equation*}
  where the coefficients $c_{jm}(t)$ belong to
  $\HH1\left(0,T; \R\right)$ for every
  $j \in \left\{1, \ldots, m\right\}$, so that, for a.e.
  $t \in [0, T]$, for all $m$, and for all $v \in \HH1(\Omega)$,
  \begin{equation}
    \label{eq:reg:weak-sol}
    \langle \dot{u}_m(t), v\rangle_* + B\left(u_m(t), v; t\right)
    = \langle f(t),v \rangle_*,
  \end{equation}
  where the bilinear term $B$ is defined in~\eqref{eq:bilinear}.  Note
  that for a.e. $t \in[0,T]$, $u_m(t)$ converges to $u(t)$ in
  $\HH1\left(\Omega\right)$ as $m \to + \infty$ and
  $\dot u_m(t) \in \HH1\left(\Omega\right)$
  for every $m \ge 1$.  Thus, substituting $v = \dot u_m(t)$
  in~\eqref{eq:reg:weak-sol} and using~\eqref{eq:bilinear} and the
  hypothesis that
  $f \in \LL2\left(0,T; \LL2\left(\Omega\right)\right)$, we get that
  \begin{equation}
    \label{eq:reg-estimate-1}
    \begin{split}
      & \quad \norm{\dot u_m(t)}_{\LL2\left(\Omega\right)}^2 +
      \underbrace{\int_\Omega \nabla u_m(t) \cdot \nabla \dot u_m(t)
        \dd x}_{I_1}
      \\
      & = \underbrace{\int_\Omega f(t) \dot u_m(t) \dd x}_{I_2} -
      \underbrace{\int_\Omega c(t,x) \, u_m(t) \, \dot u_m(t) \dd
        x}_{I_3}.
    \end{split}
  \end{equation}
  Note that
  \begin{equation}
    \label{eq:regularity_I_1}
    I_1 = \frac{1}{2} \frac{\phantom{i}\dd{\phantom{i}}}{\dd t}
    \norm{\nabla u_m(t)}_{\LL2\left(\Omega\right)}^2.
  \end{equation}
  Moreover, for a.e.~$t \in (0,T)$,
  \begin{equation}
    \label{eq:regularity_I_2}
    \begin{split}
      I_2 & \le \int_\Omega \abs{f(t)} \abs{\dot u_m(t)} \dd x \le
      \norm{f(t)}_{\LL2\left(\Omega\right)} \norm{\dot
        u_m(t)}_{\LL2\left(\Omega\right)}
      \\
      & \le 2 \norm{f(t)}_{\LL2\left(\Omega\right)}^2 + \frac{1}{8}
      \norm{\dot u_m(t)}_{\LL2\left(\Omega\right)}^2.
    \end{split}
  \end{equation}
  Finally, for a.e.~$t \in (0,T)$,
  \begin{equation}
    \label{eq:regularity_I_3}
    \begin{split}
      \abs{-I_3} & \le \int_\Omega \abs{c(t, x)} \abs{u_m(t)}
      \abs{\dot u_m(t)} \dd x
      \\
      & \le \norm{c}_{\LL\infty\left(\Omega_T\right)}
      \norm{u_m(t)}_{\LL2\left(\Omega\right)} \norm{\dot
        u_m(t)}_{\LL2\left(\Omega\right)}
      \\
      & \le 2 \norm{c}_{\LL\infty\left(\Omega_T\right)}^2
      \norm{u_m(t)}_{\LL2\left(\Omega\right)}^2 + \frac{1}{8}
      \norm{\dot u_m(t)}_{\LL2\left(\Omega\right)}^2,
    \end{split}
  \end{equation}
  provided $\norm{c}_{\LL\infty\left(\Omega_T\right)} > 0$.
  Inserting~\eqref{eq:regularity_I_1}, \eqref{eq:regularity_I_2},
  and~\eqref{eq:regularity_I_3} into~\eqref{eq:reg-estimate-1} we
  deduce that, for a.e. $t \in [0, T]$,
  \begin{align*}
    & \frac{3}{4} \norm{\dot u_m(t)}_{\LL2\left(\Omega\right)}^2
      + \frac{1}{2} \frac{\phantom{i}\dd{\phantom{i}}}{\dd t}
      \norm{\nabla u_m(t)}_{\LL2\left(\Omega\right)}^2
    \\
    \le
    & 2 \norm{f(t)}_{\LL2\left(\Omega\right)}^2
      + 2 \norm{c}_{\LL\infty\left(\Omega_T\right)}^2
      \norm{u_m(t)}_{\LL2\left(\Omega\right)}^2
  \end{align*}
  and so, integrating in time, since
  $\norma{\nabla u_m(0)}^2_{\LL2\left(\Omega\right)} \le
  \norma{g}^2_{\LL2\left(\Omega\right)}$,
  \begin{equation*}
    \begin{split}
      & \quad \frac{3}{4} \int_0^t \norm{\dot
        u_m(s)}_{\LL2\left(\Omega\right)}^2 \dd s + \frac{1}{2}
      \norm{\nabla u_m(t)}_{\LL2\left(\Omega\right)}^2
      \\
      & \le \frac{1}{2} \norm{\nabla
        u_m(0)}_{\LL2\left(\Omega\right)}^2 \!+\! 2 \int_0^t\!
      \norm{f(s)}_{\LL2\left(\Omega\right)}^2 \dd s + 2
      \norm{c}_{\LL\infty\left(\Omega_T\right)}^2 \int_0^t
      \norm{u_m(s)}_{\LL2\left(\Omega\right)}^2 \dd s
      \\
      & \le \frac{1}{2} \norm{\nabla g}_{\LL2\left(\Omega\right)}^2 +
      2 \int_0^t \norm{f(s)}_{\LL2\left(\Omega\right)}^2 \dd s + 2
      \norm{c}_{\LL\infty\left(\Omega_T\right)}^2 \int_0^t
      \norm{u_m(s)}_{\LL2\left(\Omega\right)}^2 \dd s.
    \end{split}
  \end{equation*}
  Passing to the limit as $m \to + \infty$ and
  using~\eqref{eq:L2-estimate}, we have
  \begin{equation*}
    \begin{split}
      & \quad \frac{3}{4} \int_0^t \norm{\dot
        u(s)}_{\LL2\left(\Omega\right)}^2 \dd s + \frac{1}{2}
      \norm{\nabla u(t)}_{\LL2\left(\Omega\right)}^2
      \\
      & \le \frac{1}{2} \norm{\nabla g}_{\LL2\left(\Omega\right)}^2 +
      2 \int_0^t \norm{f(s)}_{\LL2\left(\Omega\right)}^2 \dd s + 2
      \norm{c}_{\LL\infty\left(\Omega_T\right)}^2 \int_0^t
      \norm{u(s)}_{\LL2\left(\Omega\right)}^2 \dd s
      \\
      & \le \frac{1}{2} \norm{\nabla g}_{\LL2\left(\Omega\right)}^2 +
      2 \int_0^t \norm{f(s)}_{\LL2\left(\Omega\right)}^2 \dd s
      \\
      & \quad + 2 \norm{c}_{\LL\infty\left(\Omega_T\right)}^2 e^{2
        \norm{c}_{\LL\infty\left(\Omega_T\right)} t}
      \norm{g}_{\LL2\left(\Omega\right)}^2 t
      \\
      & \quad + 2 \norm{c}_{\LL\infty\left(\Omega_T\right)}^2 e^{2
        \norm{c}_{\LL\infty\left(\Omega_T\right)} t} t \int_0^t
      \norm{f(s)}_{\LL2\left(\Omega\right)}^2 \dd s.
    \end{split}
  \end{equation*}
  Since, by~\eqref{eq:L2-estimate},
  $u\in \LL\infty(0,T; \LL2(\Omega))$, the previous inequality proves
  that $u\in \LL\infty(0,T; \HH1(\Omega))$ and
  $\dot u \in \LL2(0,T; \LL2(\Omega))$.  If the boundary
  $\partial \Omega$ of $\Omega$ is of class $\CC2$, then
  $u(t) \in \HH2\left(\Omega\right)$ for a.e. $t \in [0, T]$
  (see~\cite[Theorem~8.28]{Salsa}), proving that
  $u \in \LL2(0,T; \HH2(\Omega))$.
\end{proof}



\section*{Acknowledgments}
The authors were partially supported by the 2022 GNAMPA project
\textsl{Evolution equations: well posedness, control and applications.}.

%
%

{ \bibliographystyle{abbrv}

  \bibliography{glioma.bib} }

\end{document}